\documentclass{amsart}
\usepackage{amscd}
\usepackage{latexsym,amssymb,palatino,color,graphicx,amsmath,amsthm}

\newtheorem{theorem}{Theorem}[section]
\newtheorem{lemma}[theorem]{Lemma}

\newtheorem{proposition}[theorem]{Proposition}
\newtheorem{corollary}[theorem]{Corollary}

\theoremstyle{definition}
\newtheorem{definition}[theorem]{Definition}

\theoremstyle{remark}

\newcommand{\B}{\mathcal{B}}
\renewcommand{\H}{\mathcal{H}}
\newcommand{\K}{\mathcal K}
\newcommand{\M}{\mathcal{M}}
 
\newcommand\ftwo{{\mathbb F}_2}

\newcommand\osr{{\mathcal R}}
\newcommand\oss{{\mathcal S}}
\newcommand\ost{{\mathcal T}}

\newcommand\omin{\otimes_{\textrm{min}}}
\newcommand\omax{\otimes_{\textrm{max}}}

\newcommand\oc{\otimes_{\textrm c}}

\newcommand{\A}{\mathcal{A}}                    
\newcommand{\jay}{\mathcal{J}}                   
\newcommand{\kay}{\mathcal{K}}                 
\newcommand\cstar{{\textrm C}^*}                              
\newcommand\cstare{{\textrm C}_{\textrm{min}}^*}              

\numberwithin{equation}{section}

\begin{document}

\title[Representations of Noncommutative Cubes and Prisms]{Representations of Noncommutative Cubes and Prisms}

\author[D.~Farenick, R.~Maleki, S.~Medina Varela, S.~Singla]{Douglas Farenick, Roghayeh Maleki, Sofia Medina Varela, and Sushil Singla}
\address{Department of Mathematics and Statistics, University of Regina, Regina, Saskatchewan S4S 0A2, Canada}

\email{douglas.farenick@uregina.ca}
\email{roghayeh.maleki@uregina.ca}
\email{smc472@uregina.ca}
\email{sushil@uregina.ca}
 
\subjclass[2020]{46L07, 46L05, 47A80, 47L25}

\keywords{operator system, noncommutative convex set, noncommutative state,
noncommutative cube, noncommmutative prism, dilation, tensor product, joint numerical range}

\begin{abstract}
Representations of the operator system
determined by the canonical generators of the free product of two cyclic groups of order $2$ and $k$, or $d$ cyclic groups
of order $2$, are studied for the purpose of shedding light on the noncommutative geometry of
noncommutative $d$-cubes and $k$-prisms. By way of the duality of the categories ${\textrm{NCConv}}$ and ${\textrm{OpSys}}$
of noncommutative convex sets and operator systems, respectively, an analysis of noncommutative extreme points,
exactness, the lifting property, 
automatic complete positivity, controlled completely positive extensions, tensor products, and operator system duality is undertaken.
Of note is the pairing of two classical dilation theorems of Halmos and Mirman   
to give a complete description of the
noncommutative triangular prism in terms of joint unitary dilations.
\end{abstract}

\maketitle


 \section{Introduction}

In this article, separate, but categorically related objects, will be considered in tandem. In one category, these are noncommutative convex sets, while in the other category they are operator systems. The two relevant categories, ${\textrm{NCConv}}$ and ${\textrm{OpSys}}$, described below, are dual to each other; hence, we study objects that are sometimes viewed geometrically and sometimes algebraically. 

The underlying classical geometric objects considered in this paper are cubes $\mbox{\textrm C}(d)$ in $\mathbb R^d$, with $d\geq 2$, which are formed by Cartesian products of $d$ copies of the interval $[-1,1]$, and prisms $\mbox{\textrm P}(k)$ in $\mathbb R^3$, which arise from the Cartesian product of the convex hull of the $k$-th roots of unity in $\mathbb C$ and the interval $[-1,1]$. Two special cases, enjoying additional properties, are the square $\mbox{\textrm C}(2)$ and the triangular prism $\mbox{\textrm P}(3)$. In moving to noncommutative geometry, we shall be considering the following counterparts to the classical geometric objects: noncommutative cubes and prisms, 
which are (i) graded sets in the category ${\textrm{NCConv}}$ such that their first levels in the grading 
are given by $\mbox{\textrm C}(d)$ and $\mbox{\textrm P}(k)$, respectively, or (ii) operator subsystems of (full) group C$^*$-algebras $\cstar(\mathbb Z_2*\cdots*\mathbb Z_2)$ and $\cstar(\mathbb Z_k * \mathbb Z_2)$, respectively, determined by the canonical unitary generators, where $G_1*G_2$ denotes the free product of discrete groups $G_1$ and $G_2$. 

Two classical dilation theorems in operator theory suggest particular attention should be paid to the triangular prism $\mbox{\textrm P}(3)$ as its properties in the noncommutative setting are richer than those of the noncommutative prisms arising from $\mbox{\textrm P}(k)$ for $k\geq 4$. 
The first of these results is the Halmos dilation theorem \cite{halmos1950}, which asserts that any Hilbert space contraction $x$ has a unitary dilation $u$. If, in addition, $x^*=x$, then Halmos's construction of the dilation $u$ shows that $u^*=u$; hence, every selfadjoint contraction dilates to a symmetry (that is, to a selfadjoint unitary).
The second motivating result is Mirman's dilation theorem \cite{mirman1968}: 
if $\mbox{\textrm T}\subset\mathbb C$ is a triangle with nonempty interior, and
if $x$ is a Hilbert space operator such that $\langle x\xi,\xi\rangle\in \mbox{\textrm T}$, for every unit vector $\xi$, then $x$ has a dilation to a normal operator $y$ whose spectrum consists of the three vertices of the triangle $\mbox{\textrm T}$. 
Mirman's result does not extend to convex polygonal sets such as $\mbox{\textrm P}(k)\subset\mathbb C$, if $k\geq4$, but it does extend to simplicial sets $Q$ in $\mathbb R^d$, for all $d\geq2$ \cite{binding--farenick--li1995}.

The connection between the Halmos and Mirman theorems rests with the fact that an operator $x$ is a selfadjoint contraction if and only if $\langle x\xi,\xi\rangle\in[-1,1]$, for every unit vector $\xi$. In other words, both these results may be viewed as theorems concerning the classical numerical range of an operator. So, it is natural to consider a pairing of these theorems, which in this paper involves the triangular prism in $\mathbb R^3$ arising from the Cartesian product of the triangle determined by the cube roots of unity and the line segment $[-1,1]$.

In considering noncommutative cubes and prisms in the category ${\textrm{NCConv}}$, we shall be interested in their noncommutative extreme points, which arise from irreducible representations of the associated group C$^*$-algebra. We determine at which levels of the grading one can find these noncommutative extreme points, and uncover some of their properties. We also consider a weaker form of extremality by considering elements of noncommutative cubes or triangular prisms that generate a factor, and examine what factorial types may arise.

As objects in the category ${\textrm{OpSys}}$, we study the structure of noncommutative cubes and noncommutative
prisms by way of operator-theoretic properties, including exactness, the lifting property, 
automatic complete positivity, controlled completely positive extensions, tensor products, injectivity, and operator system duality. Some operator system tensor product identities were known previously to be equivalent to the Connes Embedding Problem (CEP), and so the resolution of the CEP in the negative distinguishes some operator system tensor products considered here involving noncommutative cubes and triangular prisms. 
Finally, in relation to recent work  \cite{harris2025,scherer2026} on the lifting property for operator systems of low dimension, 
we exhibit as an outcome of our study a four-dimensional operator system in $\mathbb C^5$ that does not have the lifting property. 
 
\section{The Objects of Study}

The monograph of Blackadar \cite{Blackadar-book} serves as a basic reference for operator algebras; the papers of Choi and Effros \cite{choi--effros1977} and Kavruk, Paulsen, Todorov, and Tomforde \cite{kavruk--paulsen--todorov--tomforde2011,kavruk--paulsen--todorov--tomforde2013} are references for the operator system theory considered here. The memoir of Davidson and Kennedy \cite{davidson--kennedy2019}
is our primary reference for noncommutative convexity.

\subsection{Categorical structures}

We denote by ${\textrm{OpSys}}$ the category whose objects are operators systems, defined in \cite{choi--effros1977}
as matrix-ordered $*$-vector spaces over $\mathbb C$ possessing a distinguished Archimedean order unit, and whose morphisms are unital completely positive (ucp) linear maps. The category whose objects are noncommutative convex sets and whose morphisms are noncommutative affine functions, as defined in  \cite{davidson--kennedy2019}, is denoted by ${\textrm{NCConv}}$.

For a complex Hilbert space $\H$, let $\B(\H)$ denote the space of bounded linear operators on $\H$. 
The \emph{noncommutative state space} of a finite-dimensional operator system $\osr$ is the graded set $S_{\textrm{nc}}(\osr)$ given by
\[
S_{\textrm{nc}}(\osr)= \bigsqcup_{n\leq\aleph_0} S_n(\osr),
\]
where $S_n(\osr)$ is the set of all ucp maps $\phi:\osr\rightarrow\B(\ell^2(n))$. Such ucp maps $\phi$
are called \emph{noncommutative states} of $\osr$.
If we disregard $S_{\aleph_0}(\osr)$, then the graded set
\[
S_{\textrm{mtrx}}(\osr)= \bigsqcup_{n\in\mathbb N} S_n(\osr)
\]
is called the \emph{matrix state space} of $\osr$. (The matrix state spaces of infinite-dimensional 
operator systems $\osr$ are defined in exactly the
same way, but the noncommutative state space of an infinite-dimensional operator system $\osr$ could involve cardinalities beyond $\aleph_0$
\cite{davidson--kennedy2019}.)

If $\mathfrak x=(x_1,\dots,x_d)$ is a $d$-tuple of elements from an arbitrary operator system $\osr$, 
and if $\mathcal O_{\mathfrak x}$ is the operator subsystem of $\osr$ spanned by $\{1,x_1,x_1^*,\dots,x_d,x_d^*\}$, where $1=1_\osr$, 
the Archimedean order unit of $\osr$, 
then the \emph{noncommutative numerical range} of $\mathfrak x$ is the graded set
\[
W_{\textrm{nc}}(\mathfrak x)= \bigsqcup_{n\leq\aleph_0} W_n(\mathfrak x),
\]
where, for $n\in\mathbb N\cup\{\aleph_0\}$,
\[
W_n(\mathfrak x)=\left\{ \left(\phi(x_1),\phi(x_2),\dots,\phi(x_d)\right)\,|\,\phi\in S_n(\mathcal O_{\mathfrak x})\right\}.
\]
If we disregard the set $W_{\aleph_0}(\mathfrak x)$, then the graded set
\[
W_{\textrm{mtrx}}(\mathfrak x)= \bigsqcup_{n\in\mathbb N} W_n(\mathfrak x),
\]
is called the \emph{matrix range} of $\mathfrak x$.

We define the \emph{noncommutative complex plane} to be the graded set 
\[
\mathbb C_{\textrm{nc}} = \bigsqcup_{n\leq\aleph_0} \B(\ell^2(n)).
\]

Noncommutative state spaces, noncommutative numerical ranges, and the noncommutative complex plane
are examples of 
noncommutative convex sets. 
Recall from \cite{davidson--kennedy2019} 
that a \emph{noncommutative convex set} in a separable operator space $\mathcal V$
is a graded set
\[
K_{\textrm{nc}} = \bigsqcup_{n\leq\aleph_0} K_n,
\]
where $K_n\subseteq \M_n(\mathcal V)$ for every $n\in\mathbb N$, and $K_{\aleph_0}$ is a 
set of infinite matrices over $\mathcal V$ with uniformly bounded
finite submatrices, such that $K_{\textrm{nc}}$ is closed under direct sums and compressions. The latter two properties are:
\begin{enumerate}
\item if $\mathcal I$ is a finite or countably infinite set, and if $n$ and $\{n_i\}_{i\in\mathcal I}$ are cardinalities no larger than $\aleph_0$,
then for every family $\{a_{n,n_i}\}_{i\in \mathcal I}$ of isometries $a_{n,n_i}:\ell^2(n_i)\rightarrow\ell^2(n)$
such that $\sum_{i\in\mathcal I}a_{n,n_i}a_{n,n_i}^*=1_n$, and every bounded collection $\{\nu_{n_i}\}_{i\in\mathcal I}$ of 
$\nu_{n_i}\in K_{n_i}$, we have $\sum_{i\in\mathcal I}a_{n,n_i}\nu_{n_i}a_{n,n_i}^* \in K_n$, where the convergence of the sums,
when $\mathcal I$ is infinite, is with respect to the strong operator topology; and
\item $b^* \nu b\in K_n$, for every $x\in K_m$ and isometry $b:\ell^2(n)\rightarrow\ell^2(m)$.
\end{enumerate}
The noncommutative convex set $K_{\textrm{nc}}$ is \emph{closed} if $\mathcal V$ is a dual operator space and $K_n$ is closed in $\M_n(\mathcal V)$, for all $n$,
and \emph{compact} if each $K_n$ is compact in $\M_n(\mathcal V)$. The noncommutative convex sets considered in this paper occur with $\mathcal V=\mathbb C^d$
or $\mathcal V=\osr^\delta$, the operator system dual of a finite-dimensional operator system $\osr$.

If $K_{\textrm{nc}}$ is a noncommutative convex set, then the
set $A_{\textrm{nc}}( K_{\textrm{nc}})$ of all continuous noncommutative 
affine functions $K_{\textrm{nc}}\rightarrow\mathbb C_{\textrm{nc}}$
carries the structure of an operator system \cite[Sections 2.5 \& 3.2]{davidson--kennedy2019}.
In particular, with the noncommutative convex set $S_{\textrm{nc}}(\osr)$, where $\osr$ is an operator system of finite dimension, 
say $\osr=\mathcal O_{\mathfrak x}$ for some $d$-tuple $\mathfrak x$, then
the categorical duality results of Webster-Winker \cite{webster--winkler1999} and Davidson-Kennedy \cite{davidson--kennedy2019} assert
\begin{equation}\label{e:cd}
A_{\textrm{nc}}\left( S_{\textrm{nc}}(\mathcal O_{\mathfrak x})\right) \cong_{{\textrm{OpSys}}} \mathcal O_{\mathfrak x} \,\mbox{ and }\,
S_{\textrm{nc}}(\mathcal O_{\mathfrak x}) \cong_{{\textrm{NCConv}}} 
W_{\textrm{nc}}(\mathfrak x).
\end{equation}

\subsection{Minimal and maximal operator systems and noncommutative convex sets}

\subsubsection{Minimal and maximal operator systems}
If $\osr$ is an operator system, then the triple $\mathfrak R=(\osr, \osr_+,e_\osr)$
is an Archimedean order unit space and an operator system $\tilde\osr$ is a \emph{realisation} of $\mathfrak R$ if
the Archimedean order unit spaces $\tilde{\mathfrak R}$ and $\mathfrak R$ coincide. As shown in \cite[\S3]{paulsen--todorov--tomforde2010},
there exist operator system realisations $\osr^{\textrm{min}}$ and $\osr^{\textrm{max}}$ of $\osr$ such that
\begin{equation}\label{e:omax-omin}
\osr^{\textrm{max}} \subseteq \osr \subseteq \osr^{\min},
\end{equation}
where the inclusions are as operator systems (in the sense that the identity maps $\osr^{\textrm{max}} \rightarrow \osr$ and $\osr  \rightarrow \osr^{\textrm{min}}$ are
completely positive). If $A(K)$ is the operator subsystem of the unital abelian C$^*$-algebra $C(K)$ consisting of continuous affine maps
$K\rightarrow\mathbb C$, where $K=S_1(\osr)$ 
with the weak$^*$-topology, then $\osr^{\textrm{min}}\cong_{\textrm{OpSys}}A(K)$ and, for every operator system operator system $\oss$,
every unital positive linear map $\oss\rightarrow\osr^{\textrm{min}}$ is completely positive.
The operator system $\osr^{\textrm{max}} $ has the property that, for each operator system $\ost$, every unital
positive linear map $\osr^{\textrm{max}} \rightarrow \ost$ is completely positive. Finally, these properties completely characterise $\osr^{\textrm{min}}$ and $\osr^{\textrm{max}}$ up to 
isomorphism in ${\textrm{OpSys}}$ \cite[Theorems 3.4 and 3.22]{paulsen--todorov--tomforde2010}, and are typically denoted in the literature by 
$\mbox{\textrm OMIN}(\osr)$ and $\mbox{\textrm OMAX}(\osr)$.

\subsubsection{Minimal and maximal noncommutative convex sets}
Our primary interest is with minimal and maximal noncommutative sets whose first level is a compact convex subset of $\mathbb R^d\subset\mathbb C^d$.
Such sets were considered in \cite{davisdon--dor-on--shalit--solel2017,passer--shalit--solel2018} in the setting of matrix convexity; in the setting of
noncommutative convexity, see \cite{blecher--mcclure2026}.

If $K\subseteq\mathbb R^d$ is a compact convex set and $n\in\mathbb N\cup\{\aleph_0\}$, then 
\begin{enumerate}
\item $K_n^{\textrm{max}}$ consists of all $d$-tuples $\mathfrak a=(a_1,\dots,a_d)$ 
of selfadjoint $a_j\in\B(\ell^2(n))$ for which $W_1(\mathfrak a)\subseteq K$, and
\item $K_n^{\textrm{min}}$ consists of all $d$-tuples $\mathfrak a=(a_1,\dots,a_d)$ 
of selfadjoint $a_j\in\B(\ell^2(n))$ for which there exist a Hilbert space $\K$ containing $\ell^2(n)$ as a subspace and 
commuting selfadjoint operators $y_1,\dots, y_d\in\B(\K)$
with joint spectrum contained in $K$ such that $a_j=py_j{}_{\vert\ell^2(n)}$, for each $j$, 
where $p\in\B(\K)$ is the projection with range $\ell^2(n)$. 
\end{enumerate}

The graded sets $K^{\textrm{min}}=\displaystyle\bigsqcup_{n\leq\aleph_0} K_n^{\textrm{min}}$ and
$K^{\textrm{max}}=\displaystyle\bigsqcup_{n\leq\aleph_0} K_n^{\textrm{max}}$ are called the \emph{minimal} and \emph{maximal}
noncommutative convex sets generated by $K$, and they have the following properties:
\[
K^{\textrm{min}}\subseteq K_{\textrm{nc}} \subseteq K^{\textrm{max}} \,\mbox{ and }\, K_1^{\textrm{min}}=K_1^{\textrm{max}}=K,
\]
 for every noncommutative convex set $K_{\textrm{nc}}$ whose first graded level $K_1$ is $K$.

In general, the inclusion $K^{\textrm{min}}  \subseteq K^{\textrm{max}}$ is strict; however, it is known that
$K^{\textrm{min}}  = K^{\textrm{max}}$ if and only if $K$ is a simplex in $\mathbb R^d$ \cite{passer--shalit--solel2018}. 

\subsubsection{Minimal and maximal noncommutative state spaces}

The following connection between the two previous notions of minimality and maximality is known, but in the absence of a specific reference for the result, it is
useful to formalise it here.

\begin{proposition}\label{min-max-ncs} Let $\{e,x_1,\dots,x_d\}$ denote a linear basis of selfadjoint elements
of an operator system $\mathcal O_{\mathfrak x}$ with Archimedean order unit $e$, and
let $W_1(\mathfrak x)\subset\mathbb R^d$ denote the joint numerical range of
$\mathfrak x=(x_1,\dots,x_d)$.
We have:
\begin{enumerate}
\item $\phi\in S_{\textrm{nc}}(\mathcal O_{\mathfrak x}^{\textrm{min}})$ if and only if 
$\phi(x_1),\dots,\phi(x_d)$ have a joint dilation to commuting selfadjoint operators with
joint spectrum contained in $K$; and
\item $\phi\in S_{\textrm{nc}}(\mathcal O_{\mathfrak x}^{\textrm{max}})$ if and only if $W_1(\phi(\mathfrak x))\subseteq W_1(\mathfrak x)$.
\end{enumerate}
That is, for $K=W_1(\mathfrak x)$, we have
\[
K^{\textrm{min}}\cong_{\textrm{NCConv}} S_{\textrm{nc}}(\mathcal O_{\mathfrak x}^{\textrm{min}}) \mbox{ and }
K^{\textrm{max}}\cong_{\textrm{NCConv}} S_{\textrm{nc}}(\mathcal O_{\mathfrak x}^{\textrm{max}}).
\]
\end{proposition}

\begin{proof} Let $K=W_1(\mathfrak x)$, a compact convex subset of $\mathbb R^d$ affinely homeomorphic to the state
space $S_1(\mathcal O_{\mathfrak x})$ of $\mathcal O_{\mathfrak x}$.
From the operator system inclusions (\ref{e:omax-omin}) for $\mathcal O_{\mathfrak x}$, we deduce
\begin{equation}\label{e:max-min}
S_{\textrm{nc}}(\mathcal O_{\mathfrak x}^{\textrm{min}}) \subseteq S_{\textrm{nc}}(\mathcal O_{\mathfrak x}) \subseteq S_{\textrm{nc}}(\mathcal O_{\mathfrak x}^{\max}),
\end{equation}
where the inclusions are as graded sets. 

To prove (1), note
the operator system $A(K)$ is generated by the coordinate functions
$\varepsilon_j:K\rightarrow\mathbb C$ in which $\varepsilon_j(\zeta)=\zeta_j$, for all $\zeta=(\zeta_1,\dots,\zeta_d)\in K$. 
Because $A(K)$ is an operator subsystem of $C(K)$, every pure state on
$A(K)$ extends to a pure state on $C(K)$; therefore, the pure states of $A(K)$ are point evaluations and $S_1(A(K))$ is affinely homeomorphic to $K$.
Similar reasoning applied to pure matrix states on $A(K)$ leads to matrix states on $A(K)$ represented as matrix convex combinations of pure states
\cite{farenick2000}. More generally, via the Stinespring
decomposition, noncommutative states on $A(K)$ arise as compressions of unital $^*$-representations of $C(K)$, we deduce all noncommutative
states $\phi$ on $\mathcal O_{\mathfrak x}^{\textrm{min}}$ arise from ucp maps on $C(K)$. Restated in the language of dilation theory,
$\phi\in S_{\textrm{nc}}(\mathcal O_{\mathfrak x}^{\textrm{min}})$ if and only if $\phi(x_1),\dots,\phi(x_d)$ have a joint dilation to commuting selfadjoint operators with
joint spectrum contained in $K$.

To prove (2), it is immediate
that $W_1(\phi(\mathfrak x))\subseteq W_1(\mathfrak x)$, for each $\phi\in S_{\textrm{nc}}(\mathcal O_{\mathfrak x})$. Conversely, if $\mathfrak y$ is a $d$-tuple of selfadjoint elements $y_j$ such that
$W_1(\mathfrak y)\subseteq W_1(\mathfrak x)$, 
and $\phi:\mathcal O_{\mathfrak x}\rightarrow\mathcal O_{\mathfrak y}$ is the unital linear map for which $\phi(x_j)=y_j$, for each $j$, then $\phi$ is a positive linear map. To see this, set $x_0=e$, 
assume $x=\sum_{j=0}^d\alpha_jx_j\in\mathcal O_{\mathfrak x}$ is positive, and consider $\phi(x)\in \mathcal O_{\mathfrak y}$. For each state $\delta$ on
$\mathcal O_{\mathfrak y}$, there is a state $\omega$ on $\mathcal O_{\mathfrak x}$ for which $\delta(y_j)=\omega(x_j)$, by the hypothesis $W_1(\mathfrak y)\subseteq W_1(\mathfrak x)$. Thus,
\[
\delta(y)=\sum_{j=1}^d \alpha_j\delta(y_j) = \sum_{j=1}^d \alpha_j\omega(x_j) =\omega(x)\geq 0,
\]
which implies $y$ is positive in $\mathcal O_{\mathfrak y}$. Because every unital positive linear map of $\mathcal O_{\mathfrak x}^{\textrm{max}}$ is completely positive, we deduce
$\phi\in S_{\textrm{nc}}\left(\mathcal O_{\mathfrak x}^{\textrm{max}}\right)$. Hence, the graded set $K^{\textrm{max}}$ of all 
$d$-tuples $\mathfrak a$ of selfadjoint operators for which $W_1(\mathfrak a)\subseteq W_1(\mathfrak x)=K$ is such that
$K^{\textrm{max}}\cong_{\textrm{NCConv}} S_{\textrm{nc}}(\mathcal O_{\mathfrak x}^{\textrm{max}})$.
\end{proof}
 
\subsection{Noncommutative cubes}

The \emph{noncommutative $d$-cube},
as defined in \cite{farenick--kavruk--paulsen--todorov2014}, 
is the operator subsystem $\mbox{\textrm{NC}}(d)$ 
of the group C$^*$-algebra
$\cstar(\mathbb Z_2*\cdots* \mathbb Z_2)$ of the free product of $d$ copies of $\mathbb Z_2$ that is given by
\[
\mbox{\textrm{NC}}(d)=\mbox{Span}\{1,u_1,\dots,u_d\}\subseteq \cstar(\mathbb Z_2*\cdots* \mathbb Z_2),
\]
where $u_1,\dots,u_d$ are the canonical generators of $\cstar(\mathbb Z_2*\cdots* \mathbb Z_2)$ 
determined by the $j$-th copy $u_j$ of the generator of $\mathbb Z_2$
in $\mathbb Z_2*\cdots* \mathbb Z_2$. Note that these
elements $u_1,\dots,u_d$ are symmetries and each $u_j$ is of the form $u_j=2p_j-1$, for some
projection $p_j\in \cstar(\mathbb Z_2*\cdots* \mathbb Z_2)$. 
Moreover, the
elements $u_1,\dots,u_d\in \mbox{\textrm{NC}}(d)$ are universal for selfadjoint contractions 
in the sense that if $a_1,\dots,a_d$ are any selfadjoint operators acting on a complex Hilbert space $\H$, then there is a unital completely positive
linear map $\phi:\mbox{\textrm{NC}}(d)\rightarrow\B(\H)$ 
such that $\phi(u_j)=a_j$, for each $j$ \cite[Proposition 4.1]{farenick--kavruk--paulsen--todorov2018}.

A second definition of the noncommutative $d$-cube is given by
the graded set 
 \[
[-1,1]^d_{\textrm{nc}}=\bigsqcup_{n\leq\aleph_0} [-1,1]^d_n,
\]
where, for $n\in\mathbb N\cup\{\aleph_0\}$,
\[
[-1,1]^d_n=\left\{ (a_1,\dots,a_d)\in\B(\ell^2(n))\,|\,a_j^*=a_j\mbox{ and }\|a_j\|\leq 1, \, 1\leq j\leq d \right\}.
\]

These two definitions are related via:

\begin{theorem}\label{free d-cubes} For every $d\geq 2$, we have
\[
\mbox{\textrm{NC}}(d)=\mbox{\textrm{NC}}(d)^{\textrm{max}}
\mbox{ and }\
[-1,1]^d_{\textrm{nc}}=([-1,1]^d)^{\textrm{max}},
\]
where, for $n\in\mathbb N\cup\{\aleph_0\}$,
\[
([-1,1]^d_{\textrm{nc}})_n=\left\{ (a_1,\dots,a_d)\in\B(\ell^2(n))\,|\,a_j^*=a_j\mbox{ and }\|a_j\|\leq 1, \, 1\leq j\leq d \right\}.
\]
\end{theorem}

\begin{proof}
The unitary generators $u_1,\dots,u_d$ of $\cstar(*_1^d\mathbb Z_2)$ are necessarily selfadjoint
because $u_j^2=1$; thus, the joint numerical range $W_1(\mathfrak u)$ of $\mathfrak u=(u_1,\dots,u_d)$ is contained in $\mbox{\textrm C}(d)$. Conversely,
recall the
elements $u_1,\dots,u_d\in \mbox{\textrm{NC}}(d)$ are universal for selfadjoint contractions 
in the sense that if $a_1,\dots,a_d$ are any selfadjoint operators acting on a complex Hilbert space $\H$, then there is a unital completely positive
linear map $\phi:\mbox{\textrm{NC}}(d)\rightarrow\B(\H)$ 
such that $\phi(u_j)=a_j$, for each $j$ \cite[Proposition 4.1]{farenick--kavruk--paulsen--todorov2018}. Therefore, each $\lambda\in \mbox{\textrm C}(d)$
arises from a classical state on $ \mbox{\textrm{NC}}(d)$, implying $\mbox{\textrm C}(d)\subseteq W_1(\mathfrak u)$. In other words, 
$W_{\textrm{nc}}(\mathfrak u)=[-1,1]^d_{\textrm{nc}}$ and
$[-1,1]^d_{\textrm{nc}}=([-1,1]^d)^{\textrm{max}}$. The assertion $\mbox{\textrm{NC}}(d)=\mbox{\textrm{NC}}(d)^{\textrm{max}}$ follows from categorical duality 
applied to 
\[
S_{\textrm{nc}}(\mbox{\textrm{NC}}(d)) \cong_{\textrm{NCConv}} [-1,1]^d_{\textrm{nc}} =[-1,1]^d_{\textrm{nc}}{}^{\textrm{max}}\cong_{\textrm{NCConv}} 
S_{\textrm{nc}}(\mbox{\textrm{NC}}(d)^{\textrm{max}}),
\]
or from \cite[Proposition 6.13]{farenick--kavruk--paulsen--todorov2014}.
\end{proof}

The operator theoretic properties of free $d$-cubes were 
studied in \cite{farenick--kavruk--paulsen--todorov2014, farenick--kavruk--paulsen--todorov2018}, while the convex
analysis of free $d$-cubes was considered, for example, in \cite{evert--helton--klep--mccullough2018}. 

\subsection{Noncommutative prisms}  

If $k\in\mathbb N$ is such that $k\geq3$, then the \emph{classical $k$-prism}  $\mbox{\textrm P}(k)$ is given by 
\[
\mbox{\textrm P}(k)= \mbox{\textrm Conv}\,C_k \,\times\, \mbox{\textrm Conv}\,C_2 \subset \mathbb R^3,
\]
where $C_k$ denotes the representation of the additive cyclic group $\mathbb Z_k$ as the multiplicative subgroup of the unit circle $S^1$ 
consisting of the $k$-th roots of unity. As with noncommutative cubes, the noncommutative convex set $\mbox{\textrm P}(k)^{\textrm{max}}$ is called
a \emph{noncommutative $k$-prism} in the category ${\textrm{NCConv}}$. Thus, a triple $(a,b,c)$ of selfadjoint elements belongs to 
the graded set $\mbox{\textrm P}(k)^{\textrm{max}}$ if and only if the joint numerical range of $(a,b,c)$ is contained in the classical $k$-prism
$\mbox{\textrm P}(k)$.

The analogue of a noncommutative $k$-prism in the operator system category
is slightly different from $\mbox{\textrm P}(k)^{\textrm{max}}$ in which the underlying linear structure is not $3$-dimensional over $\mathbb R$.
In the group C$^*$-algebra $\cstar(\mathbb Z_k*\mathbb Z_2)$, let $w$ and $v$ denote the canonical unitaries that arise from the generators 
of the cyclic groups $\mathbb Z_k$ and $\mathbb Z_2$, respectively. Thus, $w^k=v^2=1$ and $w^*=w^{k-1}$. Define
\[
\mbox{\textrm{NCP}}(k)=\mbox{\textrm Span}\{1,w,w^2,\dots,w^{k-1},v\}\subseteq\cstar(\mathbb Z_k*\mathbb Z_2),
\]
which we call a \emph{noncommutative $k$-prism}. In the case $k=3$, the operator system 
\[
\mbox{\textrm{NCP}}(3)=\mbox{\textrm Span}\{1,w,w^2,v\}\subseteq\cstar(\mathbb Z_3*\mathbb Z_2)
\]
is called a \emph{noncommutative triangular prism}.

\begin{theorem}\label{ncp relations}
For every $k\geq 3$, 
\[
\mbox{\textrm P}(k)^{\textrm{min}} \subseteq W_{\textrm{nc}}(w,v) \subseteq \mbox{\textrm P}(k)^{\textrm{max}},
\]
where $w$ and $v$ are the canonical unitaries in $\cstar(\mathbb Z_k * \mathbb Z_2)$ that arise from the generators 
of the cyclic groups $\mathbb Z_k$ and $\mathbb Z_2$, respectively. 
\end{theorem}

In light of Theorem \ref{ncp relations}, we may view the noncommutative numerical  range $W_{\textrm{nc}}(w,v)$ as a noncommutative $k$-prism in the category 
${\textrm{NCConv}}$. Specifically, we define $\mbox{\textrm P}(k)_{\textrm{nc}}$ by
\begin{equation}\label{e:nck}
\mbox{\textrm P}(k)_{\textrm{nc}} = W_{\textrm{nc}}(w,v),
\end{equation}
where $w$ and $v$ are the canonical unitaries in $\cstar(\mathbb Z_k * \mathbb Z_2)$ that arise from the generators 
of the cyclic groups $\mathbb Z_k$ and $\mathbb Z_2$. Note that if $k\geq 4$, $\mbox{\textrm P}(k)_{\textrm{nc}}$ is not the noncommutative state space
of $\mbox{\textrm{NCP}}(k)$ because the state space of $\mbox{\textrm{NCP}}(k)$ is affinely homeomorphic to a compact subset of $\mathbb R^k$
rather than of $\mathbb R^3$.

The proof of Theorem \ref{ncp relations} is deferred to Section \S\ref{deferred proofs}. However, a much stronger conclusion holds for $k=3$,
as indicated by the following pairing of the theorems of Halmos and Mirman, which is proved in Section \S\ref{deferred proofs}.

\begin{theorem}[Halmos-Mirman]\label{halmos-mirman}
$\mbox{\textrm P}(3)^{\textrm{max}} = W_{\textrm{nc}}(w,v) \cong_{{\textrm{NCConv}}} S_{\textrm{nc}}(\mbox{\textrm{NCP}}(3))$.
\end{theorem}

Theorem \ref{halmos-mirman} shows that any pair of operators $x,y\in\B(\H)$, for which the numerical range of $x$ is contained in the triangle  $\mbox{\textrm Conv}\,C_3$ and $y$ is 
a selfadjoint contraction, is the ucp image of the canonical unitary generators $w$ and $v$ of $\cstar(\mathbb Z_3*\mathbb Z_2)$.

\subsection{C$^*$- and injective covers}

By the Choi-Effros Theorem \cite{choi--effros1977}, every abstract operator system $\osr$ can be realised as an operator subsystem of $\B(\H)$
for some Hilbert space $\H$. However, the C$^*$-algebras generated by different spatial realisations of a single operator system can be vastly different;
furthermore, among all such realisations, there are two that are extremal, and they are called the minimal and maximal C$^*$-covers of $\osr$.

Given an (abstract) operator system $\osr$, a \emph{C$^*$-cover} of $\osr$ is a pair $(\kappa, \A)$ consisting of a unital C$^*$-algebra $\A$
and a unital complete order embedding $\kappa:\osr\rightarrow\A$ such that $\A=\cstar\left(\kappa(\osr) \right)$.

A C$^*$-cover $(\iota,\B)$ is \emph{minimal} if, for every C$^*$-cover $(\kappa, \A)$ of $\osr$, there is a unital *-homomorphism
$\pi:\A\rightarrow\B$ such that $\iota=\pi\circ\kappa$. The minimal C$^*$-cover exists for every operator system $\osr$, and is unique
up to an isomorphism that fixes $\osr$ \cite{hamana1979b}. The minimal pair is denoted by $\left(\iota_{\textrm e},\cstare(\osr)\right)$,
and is known in the literature as the C$^*$-envelope of $\osr$.

A C$^*$-cover $(\nu,\mathcal C)$ is \emph{maximal} if, for every C$^*$-cover $(\kappa, \A)$ of $\osr$, there is a unital *-homomorphism
$\pi:\mathcal C\rightarrow\A$ such that $\kappa=\pi\circ\nu$. The maximal C$^*$-cover exists for every operator system $\osr$, and is unique
up to an isomorphism that fixes $\osr$ \cite{kirchberg--wassermann1998}. The maximal pair is denoted by 
$\left(\iota_{\textrm u},\cstar_{\textrm{max}}(\osr)\right)$,
and is known in the literature as the universal C$^*$-algebra generated by $\osr$.

The operator systems considered in this paper arise from the generators of discrete groups, and the following result describes their C$^*$-envelopes. 

\begin{theorem}\label{cstar-env} {\textrm (\cite{farenick--kavruk--paulsen--todorov2014})}
If $G$ is a discrete group generated by a finite set $\mathfrak u$ of elements of $G$, and if $\mathcal O_G$ denotes the operator subsystem of $\cstar(G)$ 
defined by
\begin{equation}\label{e:og}
\mathcal O_{G}=\mbox{\textrm Span}\{1,u,u^*\,|\,u\in\mathfrak u\}\subseteq \cstar(G),
\end{equation}
then $\cstare(\mathcal O_G)=\cstar(G)$.
\end{theorem}

An \emph{injective cover} of an operator system $\osr$ is a pair $(\kappa, \mathcal D)$ consisting of an injective C$^*$-algebra $\mathcal D$ and a unital
complete order embedding $\kappa:\osr\rightarrow\mathcal D$. Unlike C$^*$-covers, injective covers are not required to be generated by $\kappa(\osr)$.
An injective cover $(\kappa, \mathcal D)$ of $\osr$ is \emph{enveloping} if the only ucp map $\phi:\mathcal D\rightarrow\mathcal D$ for which $\kappa=\phi\circ \kappa$
is the identity map $\phi(a)=a$, for all $a\in\mathcal D$. The existence of an enveloping injective cover was established by Hamana \cite{hamana1979b}, and it is 
denoted by $(\iota_{\textrm ie},\mathcal I(\osr))$ and called the injective envelope of $\osr$. This enveloping injective cover also yields a minimal C$^*$-cover for $\osr$ via
the C$^*$-cover $\left(\iota_{\textrm ie}{}_{\vert\osr},\cstar\left(\iota_{\textrm ie}(\osr)\right)\right)$; that is, $\cstare(\osr)\cong\cstar\left(\iota_{\textrm ie}(\osr)\right)$. 

\begin{theorem}\label{inj-env} 
If $G$ is a discrete group generated by a finite set $\mathfrak u$ of elements of $G$ such that $\cstar(G)$ is primitive and antiliminal, 
then the injective envelope of the operator system 
defined by
\[
\mathcal O_{G}=\mbox{\textrm Span}\{1,u,u^*\,|\,u\in\mathfrak u\}\subseteq \cstar(G)
\]
is an AW$^*$-factor of type III.
\end{theorem}

\begin{proof} By hypothesis, $\cstar(G)$ is a separable C$^*$-algebra that has a faithful irreducible representation on a Hilbert space $\H$ such that $\cstar(G)$ contains no nonzero
compact operators. Because the injective envelope of a separable prime C$^*$-algebra is an AW$^*$-factor of type I or type III \cite[Corollary 2.4]{argerami--farenick2008}, and
since type I factors are von Neumann algebras, if the injective envelope of $\cstar(G)$ were a type I factor, then $\cstar(G)$ would contain an essential ideal consisting of nonzero
compact operators \cite{argerami--farenick2008}, contrary to the hypothesis that $\cstar(G)$ is antiliminal. Hence, $\mathcal I\left(\cstar(G)\right)$ is
a type III AW$^*$-factor. Because the injective envelopes of $\osr$ and $\cstare(\osr)$ coincide, for every operator system $\osr$, we deduce 
$\mathcal I\left(\mathcal O_G\right)$ is
a type III AW$^*$-factor.
\end{proof}

\subsection{Noncommutative extreme points, representations, and purity}

Suppose $K_{\textrm{nc}}$ is a noncommutative convex set in a separable dual operator space $\mathcal V$. 
An element $x$ is \emph{pure} in $K_{\textrm{nc}}$ if $x\in K_n$, for some $n$, and the only expressions
of $x$ of the form
\[
x=\sum_j \gamma_j^* x_j \gamma_j,
\]
where each $x_j\in K_{n_j}$ and $\{\gamma_j\}_j$ is a bounded family of operators $\ell^2(n)\rightarrow\ell^2(n_j)$
such that $\displaystyle\sum_j \gamma_j^*  \gamma_j = 1$ (the identity operator on $\ell^2(n)$),
are those expressions in which 
\begin{enumerate}
\item each $\gamma_j$ is a scalar multiple of an isometry $\nu_j$, 
\item $\nu_j^*x_j\nu_j=x$. 
\end{enumerate}

An element $x$ is \emph{maximal} in $K_{\textrm{nc}}$ if $x\in K_n$, for some $n$, and the only matrices of the form
\[
\left[ \begin{array}{cc} x & z_{12} \\ z_{21} & z_{22} \end{array}\right] 
\]
in $K_{\textrm{nc}}$ are those in which $z_{12}$ and $z_{21}$ are zero, and $z_{22}\in K_{\textrm{nc}}$.

A simple way to define noncommutative extreme points is to make use of the characterisation given in \cite[Proposition 6.1.4]{davidson--kennedy2019},
which is stated below as a definition.

\begin{definition} If $K_{\textrm{nc}}$ is a noncommutative convex set in a separable dual operator space $\mathcal V$, then an element $x\in K_{\textrm{nc}}$
is a \emph{noncommutative extreme point} if $x$ is both pure and maximal.
\end{definition}

The graded set of noncommutative extreme points of $K_{\textrm{nc}}$ is denoted by $\partial_{\textrm ext}\,K_{\textrm{nc}}$, and it satisfies the Krein-Milman theorem
in cases where $K_{\textrm{nc}}$ is compact (that is, each $K_n$ is compact):
namely, $K_{\textrm{nc}}$ is the smallest compact noncommutative convex set containing $\partial_{\textrm ext}\,K_{\textrm{nc}}$.

Endowed with the bounded-weak topology, the 
noncommutative state space of an operator system $\osr$ is an example of a compact noncommutative convex set. 
In this case, the noncommutative extreme points of $S_{\textrm{nc}}(\osr)$ are the pure, maximal ucp maps. However, the situation
is much simpler for operator systems of the form $\mathcal O_G$, as defined in (\ref{e:og}), as demonstrated by the following
result.

\begin{proposition}\label{nce} If $G$ is a discrete group generated by a finite set $\mathfrak u$ of elements of $G$, 
and if $\mathcal O_G$ denotes the operator subsystem of $\cstar(G)$ 
defined by
\[
\mathcal O_{G}=\mbox{\textrm Span}\{1,u,u^*\,|\,u\in\mathfrak u\}\subseteq \cstar(G),
\]
then a noncommutative state $\phi:\mathcal O_G\rightarrow \B(\ell^2(n))$ is a noncommutative extreme point of
$S_{\textrm{nc}}(\mathcal O_G)$ if and only if there exists an irreducible representation $\pi:\cstar(G)\rightarrow \B(\ell^2(n))$
such that $\pi_{\vert \mathcal O_G}=\phi$.
\end{proposition}

\begin{proof} Recall that a boundary representation of an operator system $\osr$ is an irreducible representation $\pi$ of $\cstare(\osr)$ 
such that the ucp map $\pi_{\vert\osr}$ has a unique ucp extension, namely $\pi$, to $\cstare(\osr)$.
By \cite[Corollary 6.2.1]{davidson--kennedy2019}, a ucp map
$\phi:\mathcal O_G\rightarrow \B(\ell^2(n))$ is a noncommutative extreme point of
$S_{\textrm{nc}}(\mathcal O_G)$ if and only if there exists a boundary representation 
$\pi:\cstar(G)\rightarrow \B(\ell^2(n))$ such that $\pi_{\vert \mathcal O_G}=\phi$. 
Since boundary representations are irreducible, it suffices to show that every 
irreducible representation $\pi:\cstar(G)\rightarrow \B(\ell^2(n))$ is a boundary representation of $\mathcal O_G$.

To this end, assume $\pi:\cstar(G)\rightarrow \B(\ell^2(n))$ is an irreducible representation, and let $\phi=\pi_{\vert \mathcal O_G}$. 
Suppose $\Phi:\cstar(G)\rightarrow \B(\ell^2(n))$ is a ucp extension of $\phi$. Then $\Phi(u)=\pi(u)$ is unitary, for every $u\in\mathfrak u$.
Thus, $u$ is in the multiplicative domain for $\Phi$, which implies $\Phi$ agrees with $\pi$ on the C$^*$-algebra generated by $\mathfrak u$.
In other words, $\Phi=\pi$, which implies $\phi$ has a unique ucp extension to the irreducible representation $\pi$ of $\cstar(G)$. Hence,
$\pi$ is a boundary representation of $\mathcal O_G$.
\end{proof}

A related form of purity exists for elements of the cone $\mathcal C \mathcal P(\osr,\ost)$ of all completely positive linear maps $\osr\rightarrow\ost$,
for operator systems $\osr$ and $\ost$. Specifically, $\phi\in \mathcal C \mathcal P(\osr,\ost)$ is \emph{pure}, if for all $\vartheta,\psi\in \mathcal C \mathcal P(\osr,\ost)$
such that $\vartheta+\psi=\phi$, there exists a scalar $\lambda\in[0,1]$ such that $\vartheta=\lambda\phi$ and $\psi=(1-\lambda)\phi$. This leads to the following
result.

\begin{corollary} If $\phi:\mathcal O_G\rightarrow\B(\ell^2(n))$ is a noncommutative extreme point of the noncommutative state space of $\mathcal O_G$,
then  $\phi$ is a pure element of the cone $\mathcal C \mathcal P(\mathcal O_G,\B(\ell^2(n)))$. Conversely, if 
$\psi:\mathcal O_G\rightarrow\B(\ell^2(n))$ is a noncommutative state such that $\psi$ is pure in the cone $\mathcal C \mathcal P(\mathcal O_G,\B(\ell^2(n)))$,
then $\psi$ has a dilation to a noncommutative extreme point of the noncommutative state space of $\mathcal O_G$.
\end{corollary}

\begin{proof} By Proposition \ref{nce}, $\phi=\pi_{\vert \mathcal O_G}$, for some boundary representation 
$\pi:\cstar(G)\rightarrow  \B(\ell^2(n))$; therefore, $\pi_{\vert\mathcal O_G}$ is a pure element of 
$\mathcal C \mathcal P(\mathcal O_G,\B(\ell^2(n)))$ \cite[Proposition 2.12]{farenick--tessier2022}.

Suppose now that 
$\psi:\mathcal O_G\rightarrow\B(\ell^2(n))$ is a noncommutative state such that $\psi$ is pure in the cone $\mathcal C \mathcal P(\mathcal O_G,\B(\ell^2(n)))$.
Because $\psi$ is pure, there exists a pure ucp extension
$\Psi:\cstar(G)\rightarrow\B(\ell^2(n))$ of $\psi$ \cite[p.~180]{arveson1969}. 
As $\Psi$ is pure, it has a dilation to an irreducible representation $\pi$ of
$\cstar(G)$ on a Hilbert space $\H_\pi$ \cite[Corollary 1.4.3]{arveson1969}. 
The restriction of $\pi$ to $\mathcal O_G$ is a noncommutative extreme point and a dilation of $\psi$; hence,
$\psi$ has a dilation to a noncommutative extreme point of the noncommutative state space of $\mathcal O_G$.
\end{proof} 

The canonical embeddings of $\mbox{\textrm{NC}}(d)$ and $\mbox{\textrm{NCP}}(k)$ into their C$^*$-envelopes are pure for $k,d\geq 2$ 
\cite[Theorem 4.3]{farenick--tessier2022}, and their canonical embeddings into their injective envelopes are pure for  $k,d\geq3$
\cite[Theorem 3.2]{farenick--tessier2022} because 
$\cstar(*_1^d\mathbb Z_2)$ and $\cstar(\mathbb Z_k*\mathbb Z_2)$ are primitive \cite[Corollary 3.4]{bedos_omland2011}.

\section{Representations of the Noncommutative Cube}

The operator system $\mbox{\textrm{NC}}(d)$ and the noncommutative convex set 
$[-1,1]^d_{\textrm{nc}}$ are dual objects in the sense that
\[
[-1,1]^d_{\textrm{nc}} \cong_{{\textrm{NCConv}}} S_{\textrm{nc}}\left( \mbox{\textrm{NC}}(d)\right) \mbox{ and }
A_{\textrm{nc}}\left([-1,1]^d_{\textrm{nc}}\right)\cong_{{\textrm{OpSys}}}   \mbox{\textrm{NC}}(d) .
\]
We also have, from Theorem \ref{free d-cubes}, that
\[
\left([-1,1]^d\right)^{\textrm{max}}=[-1,1]^d_{\textrm{nc}} ,
\]
for very $d\geq 2$. However,
as $d$ moves from $d=2$ to $d\geq3$, there are some notable differences in properties of the noncommutative 
cubes $\mbox{\textrm{NC}}(d)$, as observed in 
\cite{farenick--kavruk--paulsen--todorov2014,farenick--kavruk--paulsen--todorov2018}. 
For these reasons, the case of the noncommutative square is
discussed here first and separately from the higher-order noncommutative
$d$-cubes.

\subsection{The noncommutative square}

The analysis of the noncommutative square involves the group $\mathbb Z_2 * \mathbb Z_2$ and the group
C$^*$-algebra $\cstar(\mathbb Z_2 * \mathbb Z_2)$, both of which have been well studied. 
While the main assertion concerning $\cstar(\mathbb Z_2 * \mathbb Z_2)$, Corollary \ref{d=2} below, can be derived from the results of \cite{raeburn--sinclair1989},
the linear-algebraic proof given here has the value of simplicity, and it facilitates the determination of the enveloping injective cover of $\mbox{\textrm{NC}}(2)$.

\begin{theorem}\label{thm d=2}
Assume $v_1,v_2\in\B(\H)$ are linearly independent symmetries.  
	If $\{v_1,v_2\}$ is irreducible, then $\mbox{\textrm dim}\H=2$.
\end{theorem}

\begin{proof} Decompose $\H$ as $\K_{-1}\oplus\K_{1}$, 
	where $\K_j$ is the eigenspace of $v_j$ corresponding to the eigenvalue $j\in\{-1,1\}$. Thus, $v_1$ and
	$v_2$ have the following representations as $2\times 2$ operator matrices acting on $\H=\K_{-1}\oplus\K_{1}$:
	\[
	v_1=\left[ \begin{array}{cc} -1_{\K_{-1}} &0 \\ 0 & 1_{\K_1} \end{array}\right]
	\,\mbox{ and }\,
	v_2=\left[ \begin{array}{cc} a_{-1}&y \\ y^* & a_{1} \end{array}\right],
	\]
	for some selfadjoint $a_j\in\B(\K_j)$ and bounded linear operator $y:\K_{1}\rightarrow\K_{-1}$.
	We now determine the elements $a_j$ and $y$, using the hypothesis that $\{v_1,v_2\}$ is an irreducible set.
	
	Every matrix that commutes with $v_1$ has the form $x_{-1}\oplus x_{1}$, for $x_j\in\B(\K_j)$.
	Because $v_2$ is a symmetry, $ya_{1}=-a_{-1}y$. Hence, the diagonal matrix 
	$-a_{-1}\oplus a_{1}$
	commutes with both $v_1$ and $v_2$, and so there exists
	$\lambda\in\mathbb R$ such that $a_{-1}=\lambda1_{\K_{-1}}$ and
	$a_1=-\lambda1_{\K_1}$. The diagonal entries of the matrix 
	$v_2^2=1_\H$ lead to the equations $\lambda^2 1_{\K_{-1}}+yy^*=1_{\K_{-1}}$ and
	$\lambda^2 1_{\K_{1}}+y^*y=1_{\K_{1}}$. Thus, $-1\leq\lambda\leq1$. However, $|\lambda|\not=1$, 
	for if it were true that $|\lambda|=1$, then we would necessarily have $y=0$, implying $v_2=\lambda v_1$,
	which is in contradiction to the linear independence of $v_1$ and $v_2$. Hence, $\lambda\in(-1,1)$.
	
	Since $yy^*=(1-\lambda^2)\,1_{\K_{-1}}$ and $y^*y=(1-\lambda^2)\,1_{\K_{1}}$, the element $z=\frac{1}{\sqrt{1-\lambda^2}}y$ 
	is a linear operator $z:\K_{1}\rightarrow\K_{{-1}}$ such that $zz^*=1_{\K_{-1}}$ and $z^*z=1_{\K_{1}}$; that is, $z$ is both an isometry
	and a co-isometry. Hence, $z:\K_{1}\rightarrow\K_{-1}$ is a surjective isometry, and we can assume, henceforth
	and without loss of generality, that $\K_{-1}=\K_1$, which we denote by $\K$. Thus, $v_2$ is given by
	\[
	v_2=\left[\begin{array}{cc} \lambda 1_\K & y \\ y^* & -\lambda 1_\K\end{array}\right],
	\]
	where $yy^*=y^*y=(1-\lambda)^21_\K$.
	
	The element $g=y\oplus y\in\B(\H)$ commutes with both $v_1$ and $v_2$,
	and so $y=\mu 1_\K$, for some $\mu\in\mathbb C$ such that $|\mu|^2=1-\lambda^2$. 
	Thus,
	\[
	v_2=\left[\begin{array}{cc} \lambda 1_\K & e^{i\theta}\sqrt{1-\lambda^2} 1_\K \\ e^{-i\theta}\sqrt{1-\lambda^2}1_\K & -\lambda 1_\K\end{array}\right],
	\]
	for some $\theta\in\mathbb R$.
As $v_1$ and $v_2$ are jointly unitarily equivalent to
	\[
	\bigoplus_1^{\mbox{\textrm dim}\,\K}\left[ \begin{array}{cc}-1 &0 \\ 0 & 1\end{array}\right]
	\,\mbox{ and }\,
	\bigoplus_1^{\mbox{\textrm dim}\,\K}
	\left[ \begin{array}{cc}\lambda &e^{i\theta}\sqrt{1-\lambda^2} \\ e^{-i\theta}\sqrt{1-\lambda^2} & -\lambda\end{array}\right],
	\]
	the only way for $\{v_1,v_2\}$ to be an irreducible set is for $\mbox{\textrm dim}\,\K=1$.
	Hence, $\H=\K\oplus\K$ has dimension $2$.
\end{proof}

\begin{corollary}\label{factorial_square} The irreducible representations of $\cstar(\mathbb Z_2 * \mathbb Z_2)$ are, up to unitary equivalence,
	of the form $\pi_\lambda:\cstar(\mathbb Z_2 * \mathbb Z_2)\rightarrow\M_2(\mathbb C)$, where $\lambda\in(-1,1)$, and 
	\[
	\pi_\lambda(u_1)=\left[ \begin{array}{cc}-1 &0 \\ 0 & 1\end{array}\right]\,\mbox{ and }\,
	\pi_\lambda(u_2)=\left[ \begin{array}{cc}\lambda &\sqrt{1-\lambda^2} \\ \sqrt{1-\lambda^2} & -\lambda\end{array}\right],
	\]
	or of the form $\varrho_{k}:\cstar(\mathbb Z_2 * \mathbb Z_2)\rightarrow\mathbb C$, for $k\in\{-1,1\}$, where 
	$\varrho_{-1}(u_j)=j$ and $\varrho_1(u_j)=-j$, for $j\in\{-1,1\}$.
\end{corollary}

\begin{proof} By Theorem 3.2 and
the universal property of $\cstar(\mathbb Z_2* \mathbb Z_2)$, any pair of linear independent symmetries acting irreducibly on a Hilbert space will do so
on a space of dimension 2 and give rise to an irreducible representation $\cstar(\mathbb Z_2* \mathbb Z_2)\rightarrow\M_2(\mathbb C)$. One such symmetry
will be unitarily equivalent to $\left[ \begin{array}{cc}1 &0 \\ 0 & -1\end{array}\right]$, while the other has the form 
$\left[ \begin{array}{cc}\lambda &e^{i\theta}\sqrt{1-\lambda^2} \\ e^{-i\theta}\sqrt{1-\lambda^2} & -\lambda\end{array}\right]$ for some 
$\lambda\in(-1,1)$ and $\theta\in\mathbb R$. This pair, however, is jointly 
unitarily equivalent to
	\[
	\left[ \begin{array}{cc}1 &0 \\ 0 & -1\end{array}\right]\,\mbox{ and }\,
	\left[ \begin{array}{cc}\lambda &\sqrt{1-\lambda^2} \\ \sqrt{1-\lambda^2} & -\lambda\end{array}\right],
	\]
which is independent of $\theta$. Thus, the elements $\lambda\in(0,1)$ determine all the $2$-dimensional irreducible representations of
$\cstar(\mathbb Z_2* \mathbb Z_2)$ up to unitary equivalence. Moreover, this determination is unique, as
any unitary that fixes 
	$\left[ \begin{array}{cc} -1 &0 \\ 0 & 1\end{array}\right]$ must be diagonal, and so such a unitary equivalence cannot send 
	$\left[ \begin{array}{cc}\lambda &\sqrt{1-\lambda^2} \\ \sqrt{1-\lambda^2} & -\lambda\end{array}\right]$ to  
	$\left[ \begin{array}{cc}\tilde\lambda &\sqrt{1-\tilde\lambda^2} \\ \sqrt{1-\tilde\lambda^2} & -\tilde\lambda\end{array}\right]$
	unless $\tilde\lambda=\lambda$.
\end{proof} 

\begin{corollary} The noncommutative extreme points of the noncommutative square $[-1,1]^2_{\textrm{nc}}$
occur at levels $1$ and $2$ of the grading.
\end{corollary}

\begin{corollary}\label{bd} If $\{\lambda_n\}_{n\in\mathbb N}$ is a countable dense subset of $(-1,1]$ such that $\lambda_1=1$, then the block-diagonal operators
	\[
	\tilde u_1=\bigoplus_1^\infty \left[ \begin{array}{cc}-1 &0 \\ 0 & 1\end{array}\right] \,\mbox{ and }\,
	\tilde u_2 =\bigoplus_{n=1}^\infty \left[ \begin{array}{cc}\lambda_n &\sqrt{1-\lambda_n^2} \\ \sqrt{1-\lambda_n^2} & -\lambda_n\end{array}\right]
	\]
	are universal free symmetries.
\end{corollary}

\begin{corollary}\label{d=2} If $\A$ is the unital C$^*$-algebra given by
	\[
	\A=\left\{ f:[-1,1]\rightarrow\M_2(\mathbb C)\,|\,f\mbox{ is continuous and }f(-1) \mbox{ and }f(1)\mbox{ are diagonal}\right\},
	\]
	then there exists an isomorphism $\pi:\cstar(\mathbb Z_2 * \mathbb Z_2)\rightarrow\A$ such that
	\[
	\pi(u_1)[\lambda]=\left[\begin{array}{cc} -1 & 0 \\ 0 & 1 \end{array} \right] \,\mbox{ and }\,
	\pi(u_2)[\lambda]=\left[\begin{array}{cc} \lambda & \sqrt{1-\lambda^2} \\ \sqrt{1-\lambda^2} & -\lambda \end{array} \right],
	\mbox{ for all }\lambda\in[-1,1].
	\]
\end{corollary}

\begin{proof} Using the notations established above,   
	the map that sends $\tilde u_j$ to $\pi(u_j)$ is a C$^*$-algebraic isomorphism.
\end{proof}

While the results above are essentially present in the literature, in one form or another, the following theorem concerning the injective envelope 
of the noncommutative square has not yet been noted. The proof of the result
also gives the determination of the local multiplier algebra of $\cstar(\mathbb Z_2 * \mathbb Z_2)$.

If $\A$ is any C$^*$-algebra, and if $\jay$ and $\K$ are essential ideals of $\A$ for which $\jay\subseteq\kay$, then there is a canonical embedding
of multiplier algebras $M(\kay)\rightarrow M(\jay)$. By considering the set $\mathcal E(\A)$ of essential ideals of $\A$, ordered by inclusion, one obtains a
direct limit C$^*$-algebra denoted by $M_{\textrm{loc}}(\A)$, given by 
\[
M_{\textrm{loc}}(\A) = \varinjlim_{\jay\in\mathcal E(\A)}M(\jay).
\]
The unital C$^*$-algebra $M_{\textrm{loc}}(\A)$ is called the local multiplier algebra of $\A$ \cite{pedersen1978}, 
and it has a realisation as a C$^*$-subalgebra of the enveloping injective
cover of $\A$ (indeed, sometimes $M_{\textrm{loc}}(\A)$, as we shall see below, is the enveloping injective cover of $\A$).

In the case of commutative C$^*$-algebras of the form $C_0(X)$, for a locally compact Hausdorff space $X$, the essential ideals are determined by dense open subsets $U$ of $X$,
and the multiplier algebras have the form $C(\beta U)$, where $\beta U$ is the Stone-\v Cech compactification of $U$.

In light of Corollary \ref{d=2}, our interest is in the case $X=[-1,1]$, but we shall actually use the dense open subset $Y=(-1,1)$ instead, 
since $M_{\textrm{loc}}(\A)=M_{\textrm{loc}}(\jay)$, for any essential ideal $\jay$ of $\A$.
Define a topological space $\Delta_Y$ by way of inverse topological limits:
\[
\Delta_Y =  \varprojlim_{U\in\mathcal E(Y)}\beta U,
 \]
where $\mathcal E(Y)$ is the set of all dense open subsets of $Y$. The space $\Delta_Y$ is compact, Hausdorff, and
extremely disconnected, which implies the abelian C$^*$-algebra 
\[
C(\Delta_Y)=C\left( \varprojlim_{U\in\mathcal E(Y)}\beta U\right)=\varinjlim _{U\in\mathcal E(Y)}C(\beta U)= M_{\textrm{loc}}(C(Y))
\]
is injective.

\begin{theorem}\label{injective_envelope_d=2} 
$\mathcal I\left( \mbox{\textrm{NC}}(2)\right) = C(\Delta_Y)\otimes\M_2(\mathbb C) = M_{\textrm{loc}}\left(\cstar(\mathbb Z_2 * \mathbb Z_2)\right)$.
\end{theorem}

\begin{proof} The local multiplier algebra of 
\[
\jay=\left\{ f\in\cstar(\mathbb Z_2 * \mathbb Z_2)\,|\,f(-1)=f(1)=0\right\}\cong C_0\left( (-1,1) \right)\otimes\M_2(\mathbb C)
\]
is $M_{\textrm{loc}}(\jay)=C(\Delta_Y)\otimes\M_2(\mathbb C)$ \cite{argerami--farenick--massey2009}, which is a type I AW$^*$-algebra and, hence, injective. 
Because $\jay$ is an essential ideal of $\cstar(\mathbb Z_2 * \mathbb Z_2)$, their local multiplier algebras coincide \cite{pedersen1978}. 
Hence, the operator system inclusions 
\[
\mbox{\textrm{NC}}(2) \subseteq M_{\textrm{loc}}(\cstare(\mbox{\textrm{NC}}(2))) 
\subseteq  \mathcal I\left( \cstare(\mbox{\textrm{NC}}(2))\right)
= \mathcal I\left( \mbox{\textrm{NC}}(2)\right),
\]
with $M_{\textrm{loc}}(\cstare(\mbox{\textrm{NC}}(2)))=
M_{\textrm{loc}}(\cstar(\mathbb Z_2 * \mathbb Z_2))$ injective, imply equality in the inclusions above holds throughout, by minimality of the injective
envelope.
\end{proof}

Because $C(\Delta_Y)\otimes\M_2(\mathbb C)$ has nontrivial centre, the canonical embedding
of $ \mbox{\textrm{NC}}(2)$ into its injective envelope is not pure \cite[Theorem 2.14]{farenick--tessier2022}.

 \subsection{Noncommutative cubes, with $d\geq3$}
 
The C$^*$-algebra $\cstar(*_1^d\mathbb Z_2 )$ of the free product of $d$ copies of $\mathbb Z_2$
is residually finite dimensional \cite{bedos_omland2011}; thus, one will find noncommutative extreme points of $[-1,1]^d_{\textrm{nc}}$ at infinitely many 
finite levels of the grading of $[-1,1]^d_{\textrm{nc}}$. In fact, these extreme elements occur at all levels of the grading, as explained by the next theorem.
 
\begin{theorem}\label{ncd levels}{\textrm (\cite{davis1955,evert--helton--klep--mccullough2018,kriel2019})}
If $d\geq3$, then noncommutative extreme points of $[-1,1]^d_{\textrm{nc}}$ occur at every level of the
grading.
\end{theorem}

\begin{proof}
If $\H$ is a separable Hilbert space of dimension at least three, then $\B(\H)$ is the double commutant of three projections in $\B(\H)$ \cite{davis1955}.
Since $2p-1$ is a symmetry, for every projection $p$, we deduce the existence of irreducible triples $(\tilde v_1,\tilde v_2,\tilde v_3)$ of symmetries acting on separable
Hilbert spaces of every dimension greater than three. Every such triple $(\tilde v_1,\tilde v_2,\tilde v_3)$ induces an irreducible representation $\pi$ on 
$\cstar(\mathbb Z_2*\mathbb Z_2*\mathbb Z_2)$ via $\pi(u_j)=\tilde v_j$, where $u_1,u_2,u_3$ are the canonical 
symmetries generating $\cstar(\mathbb Z_2*\mathbb Z_2*\mathbb Z_2)$. At level $2$, one may use an irreducible pair associated with the noncommutative square,
and then add an arbitrary symmetry to create an irreducible triple; similarly, the vertices of the $3$-cube $\mbox{\textrm C}(3)$
give rise to $1$-dimensional representations of  $\cstar(\mathbb Z_2*\mathbb Z_2*\mathbb Z_2)$.

For $d>3$, one creates $d$-tuples of irreducible symmetries at any given level of the grading 
by adding arbitrary symmetries to the irreducible triple of symmetries from the previous paragraph.
\end{proof}

The following result is a construction of irreducible symmetries that is very different from the symmetries considered by Davis \cite{davis1955} in the proof of Theorem \ref{ncd levels}
above. 

\begin{theorem}\label{semiclassical}
If $n = 2^m$, then there exists an irreducible set $A_0, A_1,\ldots, A_m$ of $m+1$ symmetries in $\M_n(\mathbb C)$ such that $A_0,\dots,A_{m-1}$ are pairwise
commuting. 
\end{theorem}
\begin{proof}
	Let $k$ be a  positive integer and express $k$ in its binary representation  $$k=\sum_{i=0}^{q}k_i2^i \hspace*{1.5cm}k_i\in\{0,1\},$$  
	for a certain positive integer $q$. Define $m$ diagonal matrices $A_0, A_1,\ldots, A_{m-1}$ by setting $$A_i(k,k)=(-1)^{k_i},$$
	where $A_i(k,k)$ represents the $k$th diagonal entry of $A_i$.
	It follows directly that every  $A_i$ for $0\leq i \leq {m-1}$ is a symmetry. We claim that only diagonal matrices commute with all $A_i$. Suppose that $MA_i=A_iM$ for all $i$, where  $M\in \M_n(\mathbb C)$. Take distinct $h$ and $k$ with $0\leq h,k \leq {n-1}$. Then, there exists some $i$, $0\leq i \leq {m-1}$, for which $h_i \neq k_i$. Thus, $A_i(h,h)\neq A_i(k,k)$. Since $M$ commute with $A_i$, we have 
	\begin{align*}
		(MA_i)(h,k)=(A_iM)(h,k) \,\mbox{and }\,  M(h,k)A_i(k,k)=A_i(h,h)M(h,k).	
	\end{align*}
	As $A_i(h,h)\neq A_i(k,k)$, we must have that $M(h,k)=0$. Hence, $M(h,k)=0$ whenever $h\neq k$ which shows that $M$ is diagonal.
	
	Next, define 
	\[
	H = \underbrace{
		\begin{pmatrix}
			1 & 1 \\
			1 & -1
		\end{pmatrix}
		\otimes
		\cdots
		\otimes
		\begin{pmatrix}
			1 & 1 \\
			1 & -1
	\end{pmatrix}}_{m\text{-times}}
	\]
	and set $A_m=\frac{1}{\sqrt{n}}H$, which is a symmetry.
	
	Now, let $M$ be a diagonal matrix that commutes with $A_m$, hence commutes with $H$. Choose distinct integers $h$ and $k$ where $0\leq h,k\leq n-1$. Then,
	\begin{align*}
		(MH)(h,k)=(HM)(h,k)\, \mbox{and}\,  M(h,h)H(h,k)=H(h,k)M(k,k).
	\end{align*}
	Note that all entries of $H$ are nonzero, so this implies that $M(h,h)=M(k,k)$. As this holds for all distinct $h$ and $k$, we conclude that $M$ is a scalar multiple of the identity matrix. 
	
	Hence, the common commutant of the symmetries  $A_0,A_1,\ldots,A_m$ 
	are scalar multiples of the identity matrix. 
	\end{proof}

\begin{corollary} For every $d\geq 3$, there exists a noncommutative extreme point $(a_1,\dots,a_d)$
of $[-1,1]^d_{\textrm{nc}}$ at level $2^{d-1}$ of the grading such that 
$a_1,\dots, a_{d-1}$ are pairwise commuting.
\end{corollary}

In a different direction, but in a similar spirit, Sunder \cite{sunder1988}
has constructed extremal representations of $d$-cubes in which no proper subset
of $d'<d$ coordinates yields an extremal representation of the $d'$-cube.
 
\section{Representations of Noncommutative Prisms}

Our first goal in this section is to show that the operator system $\mbox{\textrm{NCP}}(3)$ and noncommutative convex set $\mbox{\textrm P}(3)^{\textrm{max}}$
are dual objects, which is a rephrasing of the Halmos-Mirman Theorem.

\begin{theorem}[Halmos-Mirman]\label{halmos-mirman2}
$\mbox{\textrm P}(3)^{\textrm{max}} = W_{\textrm{nc}}(w,v) \cong_{{\textrm{NCConv}}} S_{\textrm{nc}}(\mbox{\textrm{NCP}}(3))$,
where $w$ and $v$ are the canonical unitaries in $\cstar(\mathbb Z_3 * \mathbb Z_2)$ that arise from the generators 
of the cyclic groups $\mathbb Z_3$ and $\mathbb Z_2$, respectively. 
\end{theorem}

\subsection{Proof of the Halmos-Mirman Theorem}\label{deferred proofs}

To prove the Halmos-Mirman Theorem above, some preliminary results are required, beginning with the
following, somewhat surprising, dilation theorem.

\begin{theorem}\label{joint dilation}
If operators $a$ and $b$ have dilations to unitary operators of order $k$ and $2$, respectively, then
$a$ and $b$ have a common dilation to unitary operators of order $k$ and $2$, respectively.
\end{theorem}

\begin{proof} By hypothesis, there exist a Hilbert space $\K$, a unitary operator $y$, and an isometry $z:\H\rightarrow\K$
such that $y^k=1_\K$ and $a=z^*yz$. Because $b$ is the compression of a selfadjoint unitary operator,
$b$ is a selfadjoint contraction.

Let $\tilde b=zbz^*\in\B(\K)$, which is a selfadjoint contraction, and 
consider the operators $\tilde y$ and $\tilde v$ on $\K\oplus\K$ given by
\[
\tilde y = \left [\begin{array}{cc} y & 0  \\  0 & 1_{\H^{\oplus 3}} \end{array}   \right]
\,\mbox{ and }\,
\tilde v = \left [\begin{array}{cc} \tilde b & (1 -\tilde b^2)^{1/2}  \\  (1-\tilde b^2)^{1/2} & - \tilde b \end{array}   \right].
\]
Thus, $\tilde y$ and $\tilde v$ are unitary operators of order $k$ and $2$, respectively. By the universal property of the free product, 
there is a unital *-homomorphism $\pi:\cstar(\mathbb Z_k * \mathbb Z_2)\rightarrow\B(\K\oplus\K)$
such that $\pi(w)=\tilde y$ and $\pi(v)=\tilde v$, where $w$ and $v$ are the canonical unitary generators of 
$\cstar(\mathbb Z_k * \mathbb Z_2)$.
Let $g:\K\rightarrow\K\oplus\K$ be the linear isometry that embeds $\K$ into the first direct summand of $\K\oplus\K$,
and define a ucp map 
$\phi:\cstar(\mathbb Z_k * \mathbb Z_2)\rightarrow\B(\H)$ by $\phi(x)=z^*g^*\pi(x)gz$, for all $x\in \cstar(\mathbb Z_k * \mathbb Z_2)$. Hence,
\[
\phi(w)=z^*yz=a \,\mbox{ and }\, \phi(v)= z^*\tilde b z = z^*(zbz^*)z = b,
\]
which gives a common dilation of $a$ and $b$ to $w$ and $v$.
\end{proof}

The next preliminary result was mentioned already as Theorem \ref{ncp relations}; it is restated and proved below.

\begin{theorem}\label{ncp relations2}
For every $k\geq 3$, 
\[
\mbox{\textrm P}(k)^{\textrm{min}} \subseteq W_{\textrm{nc}}(w,v)= \mbox{\textrm P}(k)_{\textrm{nc}} \subseteq \mbox{\textrm P}(k)^{\textrm{max}},
\]
where $w$ and $v$ are the canonical unitaries in $\cstar(\mathbb Z_k * \mathbb Z_2)$ that arise from the generators 
of the cyclic groups $\mathbb Z_k$ and $\mathbb Z_2$, respectively. 
\end{theorem}

\begin{proof} By the definition of minimal and maximal noncommutative convex sets,
we are only required to show that the joint numerical range of the pair $(w,v)$ is the classical $k$-prism
$\mbox{\textrm P}(k)$.

The inclusion $W_1(w,v)\subseteq \mbox{\textrm P}(k)$ is apparent from the fact that $w$ and $v$ are unitaries of order $k$ and $2$, respectively. Thus, 
$\phi(w)\in\mbox{\textrm Conv}\,C_k$ and $\phi(v)\in \mbox{\textrm Conv}\,C_2$, for every state $\phi$ on $\cstar(\mathbb Z_k * \mathbb Z_2)$.

To show the reverse inclusion, 
it is enough, by the convexity of the $W_1(w,v)$ and $\mbox{\textrm P}(k)$, to show that each extreme point of $\mbox{\textrm P}(k)$
is an element of $W_1(w,v)$. The extreme points of $\mbox{\textrm P}(k)$ are 
$(\omega^j, \pm 1)$ for $1\leq j\leq k$, where $\omega$ is a primitive $k$-th root of unity.

Let $u\in\M_k(\mathbb C)$ be the cyclic shift matrix that permutes the canonical basis vectors of $\mathbb C^k$ as follows:
\[
e_1\mapsto e_k, \, e_2\mapsto e_1,\, \dots, e_k\mapsto e_{k-1}.
\]
For each $j=1,\dots,k$, let $\tilde w_j=\omega^j u$, which is a unitary of order $k$, and let $\tilde v_+\in \M_k(\mathbb C)$ be the symmetry
\[
\tilde v_+= \left[\begin{array}{cc} 0&1 \\ 1&0 \end{array}\right] \bigoplus 1_{k-2}.
\]
The unit vector $\xi=k^{-1/2}\displaystyle\sum_{j=1}^ke_j$ is a common eigenvector for $\tilde v_+$ and all $\tilde w_j$. An application of the
vector state $\phi(x)=\langle x\xi,\xi\rangle$ on $\M_k(\mathbb C)$ to all pairs $(\tilde w_j, \tilde v_+)$ produces all extreme points of $\mbox{\textrm P}(k)$
that are of the form $(\omega^j,1)$. To obtain extreme points of the form $(\omega^j,-1)$, one applies $\phi$ to all pairs $(\tilde w_j, \tilde v_-)$,
where $ \tilde v_-=- \tilde v_+$. Because these matrix pairs are composed of unitaries of order $k$ and $2$, each pair 
$(\tilde{w}_j,\tilde{v}_{\pm})$ is determined by 
some representation $\pi:\cstar(\mathbb Z_k*\mathbb Z_2)\rightarrow\M_k(\mathbb C)$ sending $w$ to $\tilde w_j$ and $v$ to
$\tilde{v}_{\pm}$; thus, the composition $\phi\circ\pi$ yields a state $\psi$ on $\cstar(\mathbb Z_k*\mathbb Z_2)$ such that
$\left(\psi(w),\psi(v)\right)=(\omega^j,\pm 1)$. Hence, $\mbox{\textrm P}(k)\subseteq W_1(w,v)$.
\end{proof}

We can now prove Halmos-Mirman Theorem (Theorem \ref{halmos-mirman}).

\begin{proof} Theorem \ref{ncp relations2} shows that $W_1(w,v)= \mbox{\textrm P}(3)$ and $W_{\textrm{nc}}(w,v)\subseteq \mbox{\textrm P}(3)^{\textrm{max}}$.
To prove $\mbox{\textrm P}(3)^{\textrm{max}}\subseteq W_{\textrm{nc}}(w,v)$, fix $n\in\mathbb N\cup\{\aleph_0\}$ and
select $(a,b)\in\mbox{\textrm P}(3)_n$. By definition, $W_1(a)\subseteq \mbox{\textrm Conv}\,C_3$ and $W_1(b)\subseteq\mbox{\textrm Conv}\,C_2$.
The theorems of Mirman and Halmos show that $a$ and $b$ have independent dilations to unitaries of order 3 and 2, respectively; therefore,
by Theorem \ref{joint dilation}, $a$ and $b$ have a common dilation to such a pair of unitaries, which proves $(a,b)\in W_{n}(w,v)$.
\end{proof}

It is well-known that Mirman's Theorem does not extend to convex $k$-gons in $\mathbb C$ if $k>3$. The typical counterexample is as follows.
If $\epsilon>0$, then the numerical range of the matrix $y_\epsilon=\left[\begin{array}{cc} 0& 1+\epsilon \\ 0 & 0\end{array}\right]$ is a disc of radius $\frac{1+\epsilon}{2}$ with centre at the origin.
Therefore, for small enough $\epsilon$, we have $W_1(y_\epsilon)\subseteq \mbox{\textrm Conv} C_k$, for every $k\geq 4$; however, as $\|y_\epsilon\|=(1+\epsilon)>1$, $y_\epsilon$ does not
have a unitary dilation. In moving to the Halmos-Mirman Theorem,  
the triple $\mathfrak a_\epsilon=(1_2,\Re y_\epsilon, \Im y_\epsilon)$
satisfies, for sufficiently small $\epsilon>0$, $\mathfrak a_\epsilon\in \mbox{\textrm P}(k)^{\textrm{max}}$ and $\mathfrak a_\epsilon\not\in \mbox{\textrm P}(k)_{\textrm{nc}}$,
which implies the Halmos-Mirman Theorem does not extend to $k$-prisms if $k\geq 4$.

\subsection{Infinite-dimensional representations}

B\'edos and Omland \cite{bedos_omland2011} have shown that $\cstar(\mathbb Z_k * \mathbb Z_2)$ and $\cstar(*_1^d\mathbb Z_2)$, where $k,d\geq3$,
are primitive, and have
uncountably many non-unitarily equivalent irreducible representations on spaces of infinite dimension. Moreover, these C$^*$-algebras
admit faithful irreducible representations that contain no compact operators. 
Hence, by invoking Theorem \ref{inj-env} and \cite[Corollary 3.4]{bedos_omland2011}, we have the following result.

\begin{theorem}\label{bo} For $k,d\geq 3$,
there are uncountably many non-unitarily equivalent noncommutative extreme points of  
$S_{\textrm{nc}}(\mbox{\textrm{NCP}}(k))$ 
 and $[-1,1]^d_{\textrm{nc}}$
at level $\aleph_0$ of the grading, and the injective envelopes of $\mbox{\textrm{NCP}}(k)$ and $\mbox{\textrm{NC}}(d)$
are AW$^*$-factors of type III.
\end{theorem}

\subsection{Finite-dimensional representations}
A representation 
$\rho : G\rightarrow \operatorname{GL}(V)$ of a group $G$ on a finite-dimensional vector space $V$ over a field $\mathbb K$
is said to be irreducible if the only $\mathbb K$-linear maps on $V$ that commute with every $\rho(g)$, for $g\in G$, 
are scalar multiplies of the identity.

\begin{lemma}\label{representation_lemma}
	Let $k\geq 2$, and let $G$ be a finite group generated by two elements, $g$ and $h$, of orders $2$ and $k$, respectively. 
	If $G$ has a faithful irreducible representation 
	on a vector space $V$ of dimension $n$, then the matrix algebra $\M_n(\mathbb C)$ is generated by two unitaries
	of order $2$ and $k$.
\end{lemma}
\begin{proof}
	Let  $\rho : G\rightarrow \operatorname{GL}(V)$ be an irreducible representation of $G$, where $V$ is of dimension $n$. 
	As in \cite[Theorem II.1.1]{simon-book}, equip $V$ with an inner product such that the images of elements of $G$ in $\mbox{\textrm End}(V)$ are unitary operators acting on the complex	
	Hilbert space $V$ with respect to this inner product, giving rise to a representation $\tilde{\rho} : G\rightarrow \M_n(\mathbb C)$ 
	such that $v=\tilde{\rho}(g)$ and $w=\tilde{\rho}(h)$ are unitary matrices. Since the representation is faithful, we have $v^2=w^k=1$; 
	because the representation $\tilde{\rho}$ is irreducible, Schur's Lemma \cite[Theorem II.4.1]{simon-book} implies the 
	commutant of the set $\{v,w\}$ is trivial. Because the adjoint of $w$ is a polynomial in $w$, the algebra generated by $\{v,w\}$ is all of $\M_n(\mathbb C)$.
	 \end{proof}

Using linear algebra and the representation theory of finite groups, we have:
	 
\begin{theorem}\label{irr repn}
For every $n\geq1$ and $k\geq3$, the group $\mathbb Z_k*\mathbb Z_2$ admits an irreducible unitary representation
$\pi_n:\mathbb Z_k*\mathbb Z_2\rightarrow \mathcal U(n)$.
\end{theorem}

\begin{proof} If $n\leq k$, then let $W\in\mathcal U(n)$ be the diagonal matrix with diagonal entries $1,\zeta,\dots,\zeta^{k-1}$, where
$\zeta$ is a primitive $k$-th root of unity. Let $P\in\M_n(\mathbb C)$ be the matrix with each entry $1/n$, so that 
$V=1_n-2P\in\mathcal U(n)$ and $V^2=1_n$. If a matrix $X$ commutes with $W$, then $X$ is diagonal. If a diagonal matrix
commutes with $P$,
then all diagonal entries coincide. Thus, the matrices that commute with both $W$ and $V$ are scalar matrices, which proves the
homomorphism $\pi_n:\mathbb Z_k*\mathbb Z_2\rightarrow \mathcal U(n)$ in which $\pi_n(v)=V$ and $\pi_n(w)=W$ is
irreducible.
Suppose now that $k<n< 2k$, and let
$\{e_1,\dots,e_n\}$ denotes the standard orthonormal basis of $\mathbb C^n$. Set 
$W=U\oplus Z\in\mathcal U(n)$, where $U\in\mathcal U(k)$ is the standard circulant unitary (that shifts
the orthonormal basis of $\mathbb C^k$), and $Z\in\mathcal U(r)$ is the diagonal matrix with entries $\zeta,\dots,\zeta^r$, for a primitive
$k$-th root of unity $\zeta$. Let
$V\in\mathcal U(n)$ be the symmetry with $j$-th column given by $e_j$,
for all columns $j$ except columns $1+r$, for $r=1,\dots,n-k$, which are given respectively by $e_{k+r}$, 
and columns $k+r$, for $r=1,\dots,n-k$, which are given by $e_{1+r}$, respectively.
A matrix $X$ that commutes with $W$ has the form $X=C\oplus \bigoplus_{r=1}^{n-k} y_r$, 
where $C$ is a $k\times k$ circulant and $y_j\in\mathbb C$. Such
a matrix $X$ commutes with $V$ only if $y_1=\dots=y_{n-k}=\frac{1}{k}\mbox{Trace}(C)$ and the off-diagonal entries of $C$ are zero. 
Thus, the matrices that commute with both $W$ and $V$ are scalar matrices, which proves the
homomorphism $\pi_n:\mathbb Z_k*\mathbb Z_2\rightarrow \mathcal U(n)$ in which $\pi_n(v)=V$ and $\pi_n(w)=W$ is
irreducible. 

For the remaining cases of $n$, namely those $n$ for which $2k\leq n$, express $k$ as
$k=2^m\ell $, where $m\geq 0$ and $\ell\geq 1$ is an odd integer. Let $A_n$ denote the alternating group.
Because $n\geq 2k=2^{m+1}\ell$ implies $n\geq 2^m+\ell+2$,  
the element
\[
(1,\dots, 2^m)(2^m+1,\dots, 2^m+\ell)(a,b)\in A_n
\]
is of order $2^m\ell=k$, for $a,b> 2^m+\ell$. Thus, for such $n$, 
$A_n$ is $(2,k)$-generated \cite{miller1928}. Because $A_n$ is simple for $n\geq 5$,
in the cases for which $\ell\geq 3$ we have $n>6$; therefore,
every representation of $A_n$ is faithful, including the irreducible representations
of $A_n$ on vector spaces of dimension $n-1$ given in \cite{fulton-harris-book}. Thus,
Lemma \ref{representation_lemma} yields an irreducible unitary representation of $A_n\rightarrow\mathcal U(n-1)$.

In the case $\ell=1$, we dispense with the consideration of  $m=0$ and $m=1$ because the irreducible 
representations of $\mathbb Z_k*\mathbb Z_2$
occur in dimensions at most $2$ when $k=0,1$; for $m\geq 2$, we have $n>2k\geq 4$, and so the simplicity of $A_n$, for $n\geq 5$, yields the required
irreducible representation in a manner analogous to the case $\ell\geq3$.
\end{proof}

\begin{corollary} For every $k\geq 3$, the noncommutative prism
$\mbox{\textrm P}(k)_{\textrm{nc}}$ and noncommutative state space $S_{\textrm{nc}}(\mbox{\textrm{NCP}}(k))$
admit noncommutative extreme points at every finite graded level.
\end{corollary}

\begin{proof} Theorem \ref{irr repn} asserts $S_{\textrm{nc}}(\mbox{\textrm{NCP}}(k))$
has noncommutative extreme points at every finite graded level. Further, if $\pi$ is a representation of $\cstar(\mathbb Z_k * \mathbb Z_2)$, then
$\{\pi(w),\pi(w)^*,\pi(v)\}$ has trivial commutant whenever $\{\pi(w),\dots,\pi(w)^{k-1},\pi(v)\}$ has trivial commutant,
as any element commuting with 
$\pi(w)$ and $\pi(v)$ will also commute with powers of $\pi(v)$ and powers of $\pi(w)$. Hence, 
$\mbox{\textrm P}(k)_{\textrm{nc}}$ also has noncommutative extreme points at every finite graded level.
\end{proof}

\subsection{Factorial representations}

If $G$ is a finitely generated discrete group with a prescribed set of generators $\mathfrak u$, and if $\mathcal O_G$
is the operator system spanned by the identity and the elements of $\mathfrak u$ and their adjoints, then a noncommutative
state $\phi:\mathcal O_G\rightarrow\B(\H)$ is \emph{factorial} if the von Neumann algebra $\phi(\mathcal O_G)''$ is a factor.
If $\phi$ extends to a representation $\pi:\cstar(G)\rightarrow\B(\H)$, then we say that $\pi$ is a \emph{factorial representation}.

Theorems \ref{ncd levels}, \ref{bo}, and \ref{irr repn} indicate $\mathcal O_G$ admits factorial representations
of type I${}_n$ and I${}_\infty$, for all $n\in\mathbb N$, if $\mathcal O_G$ is one of $\mbox{\textrm{NC}}(d)$ or $\mbox{\textrm{NCP}}(k)$, with $d,k\geq3$, .
However, these operator systems also admit factorial representations of type II${}_1$ and $\mbox{\textrm II}_\infty$.

\begin{theorem}
The operator systems $\mbox{\textrm{NC}}(d)$ and $\mbox{\textrm{NCP}}(k)$, for $d,k\geq3$, admit factorial representations of type $\mbox{\textrm II}_1$
and $\mbox{\textrm II}_\infty$.
\end{theorem}

\begin{proof} Because $\mathbb Z_k*\mathbb Z_2$, where $k\geq3$, and the free products of $d\geq3$ copies of $\mathbb Z_2$
are infinite conjugacy class groups, their group von Neumann algebras are 
$\mbox{\textrm II}_1$ factors. Hence, the unitary generators of these
group von Neumann algebras induce factorial representations from the corresponding group C$^*$-algebra, thereby yielding type-$\mbox{\textrm II}_1$
factorial representations of $\mbox{\textrm{NC}}(d)$ and $\mbox{\textrm{NCP}}(k)$, for $d,k\geq3$. Factorial representations of type $\mbox{\textrm II}_\infty$
are obtained through tensor products of type $\mbox{\textrm II}_1$ and $\mbox{\textrm I}_\infty$ factorial representations.
\end{proof}

On a related note, representations of $\mathbb Z_k* \mathbb Z_2$ in their
reduced group C$^*$-algebras were considered in \cite{paschke--salinas1979}.

\section{Quotients, Duality, and Automatic Complete Positivity}

\subsection{Representations as operator system quotients}

If $\oss$ and $\ost$ are operator systems and $\psi:\oss\rightarrow\ost$ is a ucp map, then the quotient vector space
$\oss/\ker\psi$ has the structure of an operator system \cite{kavruk--paulsen--todorov--tomforde2013}. Specifically, if $\jay=\ker\psi$
and denoting the coset of $x\in\oss$ in $\oss/\jay$ by $\dot{x}$, then 
a matrix $\dot{X}\in \M_p(\oss/\jay)$ is positive in the operator system $\oss/\jay$ if, for every $\epsilon>0$, there exists a selfadjoint
matrix $K_\epsilon\in \M_p(\jay)$ such that 
\[
\epsilon (e_\oss\otimes 1_p) + X+K_\epsilon \in \M_p(\oss)_+,
\]
and the canonical Archimedean order unit for $\oss/\jay$ is given by $\dot{e}_\oss$, where $e_\oss\in\oss$ is the Archimedean order unit for $\oss$.
If the induced ucp map 
 $\dot{\psi}:\oss/\jay\rightarrow\ost$, defined by $\dot{\psi}(\dot{x})=\psi(x)$, is a complete order isomorphism.
In this situation, $\psi$ is said to be a \emph{complete quotient map} and $\ost\cong_{\textrm{OpSys}}\oss/\jay$.

A typical use of such a complete quotient maps $\psi:\oss\rightarrow\ost$ is to pull back questions about positivity from the quotient $\oss/\jay$, where
$\jay=\ker\psi$,
to the source operator system $\oss$; in particular,
a matrix $Y\in \M_q(\oss/\jay)$ is strictly positive if and only if there is a strictly positive matrix $X\in \M_q(\oss)$
such that $Y=\dot{X}$.

The noncommutative cubes $\mbox{\textrm{NC}}(d)$ were shown, for all $d\geq2$, 
to be operator system quotients of finite-dimensional commutative C$^*$-algebras
in \cite[\S5]{farenick--kavruk--paulsen--todorov2018}. The methods used in that paper apply directly to the case of  
noncommutative prisms.  

\begin{theorem}\label{quo-thm}{\textrm (\cite[Corollary 5.3]{farenick--kavruk--paulsen--todorov2018})}
For each $k\geq 3$, there is a unital complete quotient map $\psi_{k}:\mathbb C^k \oplus \mathbb C^2 \rightarrow \mbox{\textrm{NCP}}(k)$ such that
\[
\mbox{\textrm{NCP}}(k) \cong_{\textrm{OpSys}} (\mathbb C^k \oplus \mathbb C^2) / \ker\psi_k .
\]
\end{theorem}

A second typical use is to employ complete quotient maps to analyse the operator system dual of a finite-dimensional operator system, which is our next goal.

If $\osr$ is a finite-dimensional operator system, then its dual space $\osr^\delta$ 
carries the structure of an operator system \cite[Corollary 4.5]{choi--effros1977}
by declaring a matrix $X=[f_{ij}]_{i,j=1}^q$ to be positive if the associated linear 
map $\hat X:\osr\rightarrow\M_q(\mathbb C)$, where
$\hat X(x)=[f_{ij}(x)]_{i,j=1}^q$,
is completely positive, and designating an Archimedean order unit for $\osr^\delta$ to be given 
by a faithful state on $\osr$.  Hence, 
in this regard, there is no canonical choice of an Archimedean order unit for $\osr^\delta$ in terms of the Archimedean order unit for $\osr$.

\begin{lemma}\label{dual_C^d} $(\mathbb C^d)^\delta\cong_{\textrm{OpSys}}\mathbb C^d$ with respect to  
	the Archimedean order unit of the operator system dual $(\mathbb C^d)^\delta$ 
	given by the faithful positive linear functional $\varphi:\mathbb C^d\rightarrow\mathbb C$ defined by 
	\[
	\varphi(\xi)=\sum_{j=1}^d \xi_j,\mbox{ for }\xi\in\mathbb C^d.
	\]
\end{lemma}
\begin{proof}
	Fix $q\in\mathbb N$. A matrix $X=[f_{ij}]_{i,j=1}^q\in\mathcal M_q(\mathbb C^d)$ is positive if and only if the linear map 
	$\hat{X} : \mathbb C^d\rightarrow \mathcal M_q(\mathbb C)$ given by $\hat{X}(\xi)=[f_{ij}(\xi)]_{i,j=1}^q$ is completely positive. Since $\mathbb C^d$ is an
	abelian C$^\ast$-algebra, $X$ is
 positive if and only if $\hat{X}$ is a positive map. Because the canonical basis $\{e_j\}_{j=1}^d$ determines the positive cone $(\mathbb C^d)_+$, 
 the positivity of $X$ is equivalent to the positivity of $\hat{X}(e_j)$ for all $1\leq j\leq d$. 
 Hence, there is a linear order isomorphism 
 \[
 \M_q((\mathbb C^d)^\delta)\cong_{\textrm ord}\M_q(\mathbb C)\times\dots\times\M_q(\mathbb C)= 
 \left(\displaystyle\bigoplus_1^d\M_q(\mathbb C)\right). 
 \]
 The canonical isomorphisms that identify $\displaystyle\bigoplus_1^d\M_q(\mathbb C)$, 
$\M_q(\mathbb C)\otimes\mathbb C^d$, and $\M_q(\mathbb C^d)$ are positivity preserving, and
so the linear order isomorphism above identifies the positive cones $(\M_q(\mathbb C^d)^\delta)_+$ and $(\M_q(\mathbb C^d))_+$. 	
Hence, $(\mathbb C^d)^\delta$ is completely order isomorphic to $\mathbb C^d$. To make this a unital isomorphism, let $\varphi\in(\mathbb C^d)^\delta$
be the faithful state $\varphi(\xi)=\sum_{j=1}^d \xi_j$, for $\xi\in\mathbb C^d$, and note the identification of $(\mathbb C^d)^\delta$ and $\mathbb C^d$ identifies the
faithful state $\varphi$ with the canonical identity of the C$^*$-algebra $\mathbb C^d$.
\end{proof}

The operator system dual of $\mbox{\textrm{NCP}}(k)$ is determined in the following result.

\begin{theorem}\label{dual ncp(k)} For every $k\geq 3$ and
an appropriate choice of Archimedean order units, we have
\[
\mbox{\textrm{NCP}}(k)^\delta\cong_{\textrm{OpSys}} \oss_{k,2},
\]
where $\oss_{k,2}\subseteq\mathbb C^{k+2}$ is the operator system of all tuples
$\zeta=(z_1,\dots,z_k,z_{k+1},z_{k+2})$ for which
\[
z_1 + \cdots + z_k = z_{k+1} + z_{k+2},
\]
where the matrix order on $\oss_{k,2}$ is induced by the matrix order on the abelian C$^*$-algebra $\mathbb C^{k+2}$.
\end{theorem}

\begin{proof} By
\cite[Theorem 5.1]{farenick--kavruk--paulsen--todorov2018}, 
there exists a unital linear map 
$\psi_k:\mathbb C^k\oplus\mathbb C^2\rightarrow\mbox{\textrm{NCP}}(k)$ 
with null space
\[
\ker\psi_k=\left\{\lambda\left(\sum_{j=1}^k e_j - \sum_{\ell=1}^2 e_{k+\ell}\right)\,|\,\lambda\in\mathbb C\right\}
\]
such that, as a map of operator systems, $\psi_k$ is a complete quotient map. 
From the definition of $\psi_k$ given in \cite[Theorem 5.1]{farenick--kavruk--paulsen--todorov2018}, the linear transformation
$\psi_k^\delta$ has range $\oss_{k,2}$; and, as a map of 
finite-dimensional operator systems, $\psi_k^\delta$ is a complete order embedding \cite[Proposition 1.15]{farenick--paulsen2012}. Therefore, all
that remains is the selection of Archimedean order units so that $\psi_k^\delta$ is a unital map.

To this end,
consider the faithful positive linear functional $f_\eta\in(\mathbb C^k\oplus\mathbb C^2)^\delta$ induced 
by the vector 
\[
\eta=\frac{1}{k}\sum_{j=1}^k e_j + \frac{1}{2}\sum_{j=1}^2 e_{j+k},
\]
so that $f_\eta$ is given by $f_\eta(\xi)=\displaystyle\sum_{\ell=1}^{k+2}\eta_\ell\xi_\ell$, for $\xi\in \mathbb C^k\oplus\mathbb C^2$. Because $f_\eta$ is faithful, it serves as an
Archimedean order unit for the operator system dual $(\mathbb C^k\oplus\mathbb C^2)^\delta$. The vector $\eta$ is an element of $\eta\in \oss_{k,2}$, 
which is the range of the complete order embedding $\psi_k^\delta$. Therefore, formally
$f_\eta=\psi^\delta(\varphi)$ for a unique $\varphi\in\mbox{\textrm{NCP}}(k)^\delta$. Further, because $f_\eta$ is an Archimedean order unit if and only if $\varphi$ is, 
we designate $\varphi$ as the Archimedean order unit for $\mbox{\textrm{NCP}}(k)^\delta$, which yields $\psi_k^\delta$ as a unital complete order embedding.
\end{proof}

\subsection{Complete positivity of positive linear maps}

Recall that the maximal operator system (OMAX) structure on an operator system $\osr$ is the operator system $\osr^{\textrm{max}}$ in which
the $\osr^{\textrm{max}}=\osr$ as sets, the identity map $\osr^{\textrm{max}}\rightarrow\osr$ is completely positive, and for every operator system 
$\ost $ and every unital positive linear
map $\phi:\osr^{\textrm{max}}\rightarrow\ost$, the map $\phi$ is necessarily completely positive.

\begin{theorem}\label{omax} For every $k\geq 3$, we have
$\mbox{\textrm{NCP}}(k)=\mbox{\textrm{NCP}}(k)^{\textrm{max}}$.
\end{theorem}

\begin{proof} The operator systems $\mathcal O_w$ and $\mathcal O_v$ generated by the canonical generators $w$ and $v$ in 
$\cstar(\mathbb Z_k*\mathbb Z_2)$ are the abelian C$^*$-algebras
 $\mathcal O_w=\cstar(\mathbb Z_k)\simeq\mathbb C^k$ and $\mathcal O_2=\cstar(\mathbb Z_2)\simeq\mathbb C^2$.
As noted in \cite[\S5]{farenick--kavruk--paulsen--todorov2018}, $\mbox{\textrm{NCP}}(k)\cong_{\textrm{OpSys}}\cstar(\mathbb Z_k)\oplus_1 \cstar(\mathbb Z_2)$,
the operator system coproduct of the operator systems $\cstar(\mathbb Z_k)$ and 
$\cstar(\mathbb Z_2)$.
By the universal property of operator system coproducts \cite{farenick--kavruk--paulsen--todorov2018},
if $\iota_k$ and $\iota_2$ denote the canonical embeddings of 
$\cstar(\mathbb Z_k)$ and 
$\cstar(\mathbb Z_2)$ into $\cstar(\mathbb Z_k)\oplus_1 \cstar(\mathbb Z_2)$, 
then for every operator system $\ost$ and ucp maps $\phi_k:\cstar(\mathbb Z_k)\rightarrow\ost$ and $\phi_2:\cstar(\mathbb Z_2)\rightarrow\ost$,
there is a unique ucp map $\phi:\cstar(\mathbb Z_k)\oplus_1 \cstar(\mathbb Z_2)\rightarrow\ost$ such that $\phi_k=\phi\circ\iota_k$ and 
$\phi_2=\phi\circ\iota_2$.

Suppose now that $\psi:\cstar(\mathbb Z_k)\oplus_1 \cstar(\mathbb Z_2)\rightarrow\ost$ is a unital positive linear map. Then $\psi$ induces two positive linear maps
$\phi_k:\cstar(\mathbb Z_k)\rightarrow\ost$ and $\phi_2:\cstar(\mathbb Z_2)\rightarrow\ost$ via $\phi_k=\psi\circ\iota_k$ and 
$\phi_2=\psi\circ\iota_2$. Because $\phi_k$ and $\phi_2$ are positive linear maps of abelian C$^*$-algebras, they are completely
positive. Thus, by the universal property of the coproduct, the map $\psi$ is necessarily completely positive; hence,
$\cstar(\mathbb Z_k)\oplus_1 \cstar(\mathbb Z_2)=(\cstar(\mathbb Z_k)\oplus_1 \cstar(\mathbb Z_2))^{\textrm{max}}$.
\end{proof}

\section{Tensor Products}

The theory of operator system tensor products \cite{kavruk--paulsen--todorov--tomforde2011}
is inspired by the corresponding theory for C$^*$-algebras, although there are some striking differences between the theories. 

Henceforth, if
$\osr_1$ and $\ost_1$ are operator subsystems of $\osr_2$ and $\ost_2$, respectively, and if $\otimes_\sigma$ and $\otimes_\tau$ are
operator system tensor products, then the notation 
\[
\osr_1\otimes_\sigma\ost_1 \,\subseteq\, \osr_2\otimes_\tau\ost_2
\]
indicates that the canonical algebraic embedding $\iota:\osr_1\otimes\ost_1\rightarrow\osr_2\otimes\ost_2$ of algebraic tensor product spaces is 
a completely positive linear map $\osr_1\otimes_\sigma\ost_1 \rightarrow \osr_2\otimes_\tau\ost_2$.
The stronger notation, 
\[
\osr_1\otimes_\sigma\ost_1 \,\subseteq_{\textrm{coi}}\, \osr_2\otimes_\tau\ost_2,
\]
indicates that the canonical algebraic embedding $\iota:\osr_1\otimes\ost_1\rightarrow\osr_2\otimes\ost_2$ of algebraic tensor product spaces is 
a complete order embedding $\osr_1\otimes_\sigma\ost_1 \rightarrow \osr_2\otimes_\tau\ost_2$. In this latter setting, if a
matrix $X\in\M_q(\osr_1\otimes_\sigma\ost_1)$ is positive in $\M_q(\osr_2\otimes_\tau\ost_2)$, then it is also positive in $\M_q(\osr_1\otimes_\sigma\ost_1)$.
Thus, for two operator system tensor product structures $\otimes_\sigma$ and $\otimes_\tau$, we write
$\osr\otimes_\sigma\ost= \osr\otimes_\tau\ost$ if
\[
\osr\otimes_\sigma\ost\,\subseteq_{\textrm{coi}}\, \osr\otimes_\tau\ost
\,\mbox{ and }\,
\osr\otimes_\tau\ost\,\subseteq_{\textrm{coi}}\, \osr\otimes_\sigma\ost.
\]

Recall from \cite{kavruk--paulsen--todorov--tomforde2011} that
there exist minimal and maximal operator system tensor product structures $\omin$ and $\omax$ 
such that, for any operator system tensor product structure
 $\otimes_\sigma$, and any operator systems $\osr$ and $\ost$, we have
 \[
 \osr\omax\ost\,\subseteq\,\osr\otimes_\sigma\ost\,\subseteq\,\osr\omin\ost.
 \]
 The main result of this section shows that, if $\osr,\ost\in\{\mbox{\textrm{NC}}(d),\mbox{\textrm{NCP}}(k)\,|\, d,k\geq 2\}$, then
 \[
 \osr\omin\ost \not= \osr\omax\ost.
 \]
 The distinction between $\osr\omin\ost$ and $\osr\omax\ost$ above 
 is known if both $\osr$ and $\ost$ are noncommutative cubes \cite{farenick--kavruk--paulsen--todorov2014}, and so the
 novelty here occurs when one or both of the tensor factors is $\mbox{\textrm{NCP}}(k)$. To show this, it is useful to use C$^*$-algebra tensor products.
 
 If $\A$ and $\B$ are unital C$^*$-algebras, then the operator system tensor products $\A\omin\B$ and $\A\omax\B$ of $\A$ and $\B$
 are closely related to the minimal and maximal C$^*$-algebra tensor products, which have the same notation, in which the norm completions of
 the operator systems $\A\omin\B$ and $\A\omax\B$ result in the C$^*$-algebras $\A\omin\B$ and $\A\omax\B$. Therefore, in what follows,
 the notation $\A\omin\B$ and $\A\omax\B$ always refers to the C$^*$-algebra tensor products.
 
As in most works involving operator system tensor products,
the commuting operator system tensor product structure, denoted by $\oc$,
has a substantial role here. Recall that a matrix $X\in M_q(\osr\oc\ost)$ is positive, if for all ucp maps
$\phi:\osr\rightarrow\B(\H)$ and $\psi:\ost\rightarrow\B(\H)$ with commuting ranges, the matrix $(\phi\cdot\psi)^{[q]}(X)$ is positive in 
$\M_q(\B(\H))$, where $\phi\cdot\psi$ is the linear map on $\osr\otimes\ost\rightarrow\B(\H)$ satisfying $(\phi\cdot\psi)(x\otimes y)=\phi(x)\psi(y)$, for all
$x\in\osr$ and $y\in\ost$.  An important feature of this tensor product is that $\osr\oc\A=\osr\omax\A$, whenever $\A$ is a unital C$^*$-algebra.
The determination of $\osr\oc\ost$ is often difficult; as shown in \cite{farenick--kavruk--paulsen--todorov2018,kavruk2014}, basic questions
about $\oc$ can sometimes be logically equivalent to questions regarding the Connes Embedding Problem.
 
 \subsection{Some general propositions}
 
 As defined earlier in equation (\ref{e:og}), 
 $\mathcal O_{G_j}$ shall denote the operator subsystem of $\cstar(G_j)$ determined by the unitary generators, and their adjoints, of a fixed finite
 set of generators of a discrete group $G_j$.
 
 The commuting tensor product $\mathcal O_{G_1}\otimes_{\textrm c}\mathcal O_{G_2}$ is induced by the complete
 order embedding 
 \[
 \mathcal O_{G_1}\otimes_{\textrm c}\mathcal O_{G_2} \,\subseteq_{\textrm{coi}}\,\cstar_{\textrm{max}}(\mathcal O_{G_1}) \omax \cstar_{\textrm{max}}(\mathcal O_{G_2}).
 \]
Therefore, the
case in which
\[
\mathcal O_{G_1}\otimes_{\textrm c}\mathcal O_{G_2} \,\subseteq_{\textrm{coi}}\,\cstar_{\textrm{min}}(\mathcal O_{G_1}) \omax \cstar_{\textrm{min}}(\mathcal O_{G_2}) 
\]
is quite special and has strong consequences, as the following proposition shows.

\begin{proposition}\label{(c-max)_unitaries}
	If $\mathcal O_{G_1}\otimes_{\textrm c}\mathcal O_{G_2}\subseteq_{\textrm{coi}}\cstar(G_1)\otimes_{\textrm{max}}\cstar(G_2)$, 
	then $u_1\otimes u_2$ is unitary in $\cstar(G_1)\otimes_{\textrm{max}}\cstar(G_2)$ for each of the finitely-many
	unitary generators $u_j\in\mathcal O_{G_j}$ of $G_j$, and
\[
\cstar_{\textrm{min}}(\mathcal O_{G_1}\otimes_{\textrm c}\mathcal O_{G_2})=\cstar(G_1)\otimes_{\textrm{max}}\cstar(G_2).
\]
\end{proposition}
\begin{proof} As defined in \cite{farenick--kavruk--paulsen--todorov2014},
the operator system tensor product $\otimes_{\textrm ess}$ is given by the inclusion of $\osr\otimes\ost$ into $\cstare(\osr)\omax\cstare(\ost)$.
Since $\cstar(G_j)=\cstare(\mathcal O_{G_j})$, the hypothesis $\mathcal O_{G_1}\otimes_{\textrm c}\mathcal O_{G_2}\subseteq_{\textrm{coi}}\cstar(G_1)\otimes_{\textrm{max}}\cstar(G_2)$ 
implies $\mathcal O_{G_1}\otimes_{\textrm ess}\mathcal O_{G_2}=\mathcal O_{G_1}\otimes_{\textrm c}\mathcal O_{G_2}$, which in turn implies 
$\cstar(G_1)\otimes_{\textrm{max}}\cstar(G_2)=\cstare(\mathcal O_{G_1}\otimes_{\textrm c}\mathcal O_{G_2})$ and
that
$u_1\otimes 1_{G_2}$ and $1_{G_1}\otimes u_2$ are unitaries in $\cstar(G_1)\otimes_{\textrm{max}}\cstar(G_2)$,
for each of the finitely-many
unitary generators $u_j\in\mathcal O_{G_j}$ of $G_j$
\cite[Theorem 3.5]{farenick--kavruk--paulsen--todorov2014}.
Furthermore, by \cite[Lemma 3.4]{farenick--kavruk--paulsen--todorov2014}, $u_1\otimes u_2=(u_1\otimes 1_{G_2})(1_{G_1}\otimes u_2)$, 
and so $u_1\otimes u_2$ is unitary in $\cstar(G_1)\otimes_{\textrm{max}}\cstar(G_2)$. 
\end{proof}

\begin{proposition}\label{rdci-(min,max)}
	If $\mathcal O_{G_j}\otimes_{\textrm c}\mathcal R\subseteq_{\textrm{coi}}\cstar(G_j)\otimes_{\textrm c}\mathcal R$ for all operator systems $\mathcal R$ and $j=1,2$, then $$\mathcal O_{G_1}\otimes_{\textrm{min}}\mathcal O_{G_2}=\mathcal O_{G_1}\otimes_{\textrm c}\mathcal O_{G_2}\implies \cstar(G_1)\otimes_{\textrm{min}}\cstar(G_2)=\cstar(G_1)\otimes_{\textrm{max}}\cstar(G_2).$$
\end{proposition}
\begin{proof}
	Assume $\mathcal O_{G_1}\otimes_{\textrm{min}}\mathcal O_{G_2}=\mathcal O_{G_1}\otimes_{\textrm c}\mathcal O_{G_2}$.  By hypothesis, we have \begin{align*}
		\mathcal O_{G_1}\otimes_{\textrm{min}}\mathcal O_{G_2}=\mathcal O_{G_1}\otimes_{\textrm c}\mathcal O_{G_2}&\subseteq_{\textrm{coi}} \cstar(G_1)\otimes_{\textrm c}\mathcal O_{G_2}\\
		&\subseteq_{\textrm{coi}} \cstar(G_1)\otimes_{\textrm c}\cstar(G_2)\\
		&=\cstar(G_1)\otimes_{\textrm{max}}\cstar(G_2).
	\end{align*} 	
Let $\mathcal W=\{u_1\otimes u_2\,|\,u_j\in\mathcal O_{G_j}\text{ is a unitary generator of }G_j\}\subseteq \cstar(G_1)\otimes_{\textrm{min}}\cstar(G_2),$
where the identity elements of $G_1$ and $G_2$ are considered as unitary generators.
If $\iota : \mathcal O_{G_1}\otimes_{\textrm{min}}\mathcal O_{G_2}\rightarrow \cstar(G_1)\otimes_{\textrm{max}}\cstar(G_2)$ is the unital complete order embedding
that exists by the coi-inclusions above, then,
by Proposition \ref{(c-max)_unitaries},
$\iota(u_1\otimes u_2)$ is unitary in $\cstar(G_1)\otimes_{\textrm{max}}\cstar(G_2)$, for every $u_1\otimes u_2\in\mathcal W$.
Hence, $\iota$ maps the spanning set of unitaries in $\mathcal O_{G_1}\otimes_{\textrm{min}}\mathcal O_{G_2}$ 
to unitary elements of $\cstar(G_1)\otimes_{\textrm{max}}\cstar(G_2)$; thus, $\mathcal W$ is in the multiplicative domain of $\iota$. 
Hence, $\iota$ is a unital $*$-homomorphism on
\[
\cstar(\mathcal W)=\cstar(\mathcal O_{G_1}\otimes_{\textrm{min}}\mathcal O_{G_2}) = \cstare(\mathcal O_{G_1})\omin \cstare(\mathcal O_{G_2}),
\]
where the final equality is a consequence of \cite[Theorem 4.4]{kavruk--paulsen--todorov--tomforde2011}. Therefore,
we obtain a unital $\ast$-monomorphism
$\cstar(G_1)\otimes_{\textrm{min}}\cstar(G_2)\rightarrow \cstar(G_1)\otimes_{\textrm{max}}\cstar(G_2)$, which proves 
\[
\cstar(G_1)\otimes_{\textrm{min}}\cstar(G_2)=\cstar(G_1)\otimes_{\textrm{max}}\cstar(G_2),
\]
as required.
\end{proof}

With the resolution of the Connes Embedding Problem in the negative \cite{goldbring2022,nonCEP}, the C$^*$-algebras 
$\cstar(G_1)\otimes_{\textrm{min}}\cstar(G_2)$ and $\cstar(G_1)\otimes_{\textrm{max}}\cstar(G_2)$ are now known to be non-isomorphic for
many groups $G_1$ and $G_2$. Hence, as 
an immediate corollary of Proposition \ref{rdci-(min,max)}, we deduce the following result.

\begin{corollary}\label{main_corollary_tensor}
	If each of $G_1$ and $G_2$ contains the free group $\ftwo$ as a subgroup, and if $\mathcal O_{G_j}\otimes_{\textrm c}\mathcal R\subseteq_{\textrm{coi}}\cstar(G_j)\otimes_{\textrm c}\mathcal R$ for every operator system $\mathcal R$, then 
\[
\mathcal O_{G_1}\otimes_{\textrm{min}}\mathcal O_{G_2}\neq \mathcal O_{G_1}\otimes_{\textrm c}\mathcal O_{G_2}.
\]
\end{corollary}
\begin{proof}
	If, on the contrary, $\mathcal O_{G_1}\otimes_{\textrm{min}}\mathcal O_{G_2}= \mathcal O_{G_1}\otimes_{\textrm c}\mathcal O_{G_2}$, then Proposition \ref{rdci-(min,max)} 
	would imply $\cstar(G_1)\otimes_{\textrm{min}}\cstar(G_2)=\cstar(G_1)\otimes_{\textrm{max}}\cstar(G_2)$, which in turn would imply
	$\cstar(\ftwo)\omin\cstar(\ftwo)=\cstar(\ftwo)\omax\cstar(\ftwo)$ by the hypotheis $\mathbb F_2\subseteq G_j$
	\cite[Proposition 3.3]{farenick--kavruk--paulsen--todorov2018}. However, the equality of the C$^*$-algebra tensor products involving $\cstar(\ftwo)$
	is Kirchberg's equivalency to an affirmative outcome of the CEP \cite{kirchberg1993}; hence, it cannot be that 
	$\mathcal O_{G_1}\otimes_{\textrm{min}}\mathcal O_{G_2}= \mathcal O_{G_1}\otimes_{\textrm c}\mathcal O_{G_2}$.
\end{proof}

\subsection{The main result on tensor products}
Our aim is to prove the following theorem (Theorem \ref{main result tp}):
if $\osr,\ost\in\{\mbox{\textrm{NC}}(d),\mbox{\textrm{NCP}}(k)\,|\,d,k\geq3\}$, then
 \[
 \osr\omin\ost \not= \osr\oc\ost \not = \osr\omax\ost.
 \]

\begin{lemma}\label{rdci_prism}
	For every operator system $\mathcal R$, and $d,k\geq 2$,
\[
\mbox{\textrm{NC}}(d)\otimes_{\textrm c}\mathcal R\subseteq_{\textrm{coi}}\cstar(*_1^d \mathbb Z_2)\otimes_{\textrm c}\mathcal R
\,\mbox{ and }\,
\mbox{\textrm{NCP}}(k)\otimes_{\textrm c}\mathcal R\subseteq_{\textrm{coi}}\cstar(\mathbb Z_k\ast \mathbb Z_2)\otimes_{\textrm c}\mathcal R.
\] 
\end{lemma}

\begin{proof}
The assertion $\mbox{\textrm{NC}}(d)\otimes_{\textrm c}\mathcal R\subseteq_{\textrm{coi}}\cstar(*_1^d \mathbb Z_2)\otimes_{\textrm c}\mathcal R$ is proved in
\cite[Lemma 7.2]{farenick--kavruk--paulsen--todorov2014}.

Consider now the case of $\mbox{\textrm{NCP}}(k)$, with $k\geq 3$.
By definition of $\otimes_{\textrm c}$, it is sufficient to show that, for every pair of ucp maps 
$\psi : \mbox{\textrm{NCP}}(k)\rightarrow\B(\H)$ and	
$\phi : \mathcal R\rightarrow\B(\H)$ with commuting ranges, there exists a ucp extension 
$\Psi : \cstar(\mathbb Z_k\ast \mathbb Z_2)\rightarrow\B(\H)$ of $\psi$ such that $\Psi$ and $\phi$ also have commuting ranges.

As done earlier, denote the canonical unitary generators of $\mbox{\textrm{NCP}}(k)$ by $w$ and $v$, where $w^k=v^2=1$. Let
$\psi : \mbox{\textrm{NCP}}(k)\rightarrow\B(\H)$ and	
$\phi : \mathcal R\rightarrow\B(\H)$ be ucp maps with commuting ranges.
Express $\psi$ in its Stinespring decomposition: $\psi=s^*\rho s$, for some unital $*$-representation
$\rho:\cstar(\mathbb Z_k * \mathbb Z_2)\rightarrow\B(\H_\rho)$ and isometry $s:\H\rightarrow\H_\rho$. As $\rho(w)$ is a unitary of order
$k$, it admits a spectral decomposition $\rho(w)=\displaystyle\sum_{j=0}^{k-1}\omega^j p_j$, where $\omega$ is a primitive $k$-th root of
unity and each $p_j$ is a spectral projection. Thus, if $h_j=s^*p_js$, for each $j$, then
$\displaystyle\psi(w)=\sum_{j=0}^{k-1} \omega^j h_j$.

Consider the isometry $z:\H\rightarrow\H^{\oplus k}$, where $\H^{\oplus k}$ is the direct sum of $k$ copies of $\H$, given by 
$z\xi=h_0^{1/2}\xi\oplus\cdots\oplus h_{k-1}^{1/2}\xi$, for $\xi\in\H$. Let $b:\H^{\oplus k}\rightarrow\H^{\oplus k}$ be given by $b=z\psi(v)z^*$. 
Then, $b$ is a selfadjoint contraction and, therefore, admits a Halmos dilation to a symmetry $\tilde v$ on $\H^{\oplus k}\oplus\H^{\oplus k}$:
\[
\tilde v = \left[ \begin{array}{cc} b & (1-b^2)^{1/2} \\ (1-b^2)^{1/2} & -b \end{array}\right].
\]

If $y\in \B(\H^{\oplus k} )$ and $\tilde w\in \B(\H^{\oplus k}\oplus\H^{\oplus k})$ are given by
\[
y=\bigoplus_{j=0}^{k-1}\omega^j 1_\H
\,\mbox{ and }\,
\tilde w = y\oplus 1_{\H^{\oplus k}},
\]
then, $\tilde w$ is a unitary of order $k$ and, by the universal property, there exists 
a unital $*$-homomorpshim
$\pi:\cstar(\mathbb Z_k * \mathbb Z_2)\rightarrow\B(\H^{\oplus k}\oplus\H^{\oplus k})$
such that $\pi(w)=\tilde w$ and $\pi(v)=\tilde v$.
With respect to the decomposition $\H^{\oplus k}\oplus\H^{\oplus k}$, express $\pi$ as a 
$2\times2$ matrix of maps:
\[
\pi=\begin{bmatrix}
		\pi_{11}&\pi_{12}\\\pi_{21}&\pi_{22}
	\end{bmatrix}.
\]
Define a linear map $\Psi:\cstar(\mathbb Z_k * \mathbb Z_2)\rightarrow\B(\H)$ by $\Psi(x)=z^*\pi_{11}(x)z$, for $x\in \cstar(\mathbb Z_k * \mathbb Z_2)$.
Because $\pi_{11}$ is a compression of a $*$-homomorphism, $\pi_{11}$ and, hence, $\Psi$ are completely positive. Since $\Psi(w)=\psi(w)$
and $\Psi(v)=\psi(v)$, the map $\Psi$ is a completely positive extension of $\psi$ to $\cstar(\mathbb Z_k * \mathbb Z_2)$.

It remains to show that $\Psi$ and $\phi$ have commuting ranges. To this end, 
first note that, because $\rho(w)$ is a normal operator with finite spectrum, each spectral projection $p_j$ of $\rho(w)$ is a polynomial in $\rho(w)$,
say $p_j=f_j(\rho(w))$, for some $f_j\in\mathbb C[t]$. Thus, $h_j=f_j(\psi(w))$, for every $j$, and any operator $g$ that commutes 
with $\psi(w)$ also commutes with $h_j$; further, by functional calculus, $g$ also commutes with $h_j^{1/2}$.

For each $r\in\mathcal R$, let $\Phi(r)=\phi(r)\oplus\cdots\oplus\phi(r)\in \B(\H^{\oplus k})$. As $\phi(r)$ commutes with $\psi(w)$, $\phi(r)$ commutes
with each $h_j$ and, therefore, $\Phi(r)z=z\phi(r)$. Passing to adjoints in $\osr$, we also have $z^*\phi(r)=z^*\Phi(r)$, for every $r\in\osr$.

Recall $b=z\psi(v)z^*$ and, for every $r\in\osr$, $\phi(r)$ commutes with $\psi(v)$. Thus, 
\[
\Phi(r)b=\Phi(r)z\psi(v)z^*=z\phi(r)\psi(v)z^*=z\psi(v)\phi(r)z^*=z\psi(v)z^*\Phi(r)=b\Phi(r),
\]
which implies that $\Phi(r)$ commutes with $(1-b^2)^{1/2}$ and $\Phi(r)\oplus\Phi(r)$ commutes with $\tilde v$.

Note that every $\Phi(r)\oplus\Phi(r)$ also commutes with $\tilde v$. Hence, as $\pi$ is a homomorphism, each
$\Phi(r)\oplus\Phi(r)$ commutes with $\pi(a)$, for every $a\in\cstar(\mathbb Z_k * \mathbb Z_2)$. In passing to the $(1,1)$-entry of the
$2\times 2$ operator matrix $\pi(a)$, $\Phi(r)\pi_{11}(a)=\pi_{11}(z)\Phi(r)$, for all $r\in\osr$ and $a\in\cstar(\mathbb Z_k * \mathbb Z_2)$.
Thus,
\[
\phi(r)z^*\pi_{11}(a)z=z^*\Phi(r)\pi_{11}(a)z=z^*\pi_{11}(a)\Phi(r)z=z^*\pi_{11}(a)z\phi(r)
\]
yields $\phi(r)\Psi(a)=\Psi(a)\phi(r)$, for all $r\in\osr$ and $a\in\cstar(\mathbb Z_k * \mathbb Z_2)$. Hence,
$\Psi$ and $\phi$ have commuting ranges.
\end{proof}

We now are ready to prove the main result of this section. (If $\osr$ and $\ost$ are noncommutative cubes, 
then  $\osr\oc\ost \not = \osr\omax\ost$ was
established previously in \cite[Theorem 7.13]{farenick--kavruk--paulsen--todorov2014}.)

\begin{theorem}\label{main result tp}
If $\osr,\ost\in\{\mbox{\textrm{NC}}(d),\mbox{\textrm{NCP}}(k)\,|\,d,k\geq3\}$, then
 \[
 \osr\omin\ost \not= \osr\oc\ost \not = \osr\omax\ost.
 \]
\end{theorem}

\begin{proof} By Lemma \ref{rdci_prism},
every operator system $\osr\in\{\mbox{\textrm{NC}}(d),\mbox{\textrm{NCP}}(k)\,|\,d,k\geq3\}$
satisfies $\osr\oc\oss\subseteq_{\textrm{coi}}\cstare(\osr)\omax\oss$ for every operator system $\oss$.
Further, the groups $\ast_1^d\mathbb Z_2$ and $\mathbb Z_k*\mathbb Z_2$ each contain $\ftwo$ as a subgroup, when $k,d\geq3$.
Therefore, any pair of operator systems $\osr,\ost\in\{\mbox{\textrm{NC}}(d),\mbox{\textrm{NCP}}(k)\,|\,d,k\geq3\}$ satisfies the hypothesis of 
Corollary \ref{main_corollary_tensor}; hence, by Corollary \ref{main_corollary_tensor},
$\osr\omin\ost \not= \osr\oc\ost$.

An unpublished result of Kavruk \cite[Lemma 6.7]{kavruk2012} shows there exist unital complete order embeddings
$\iota_\osr:\mbox{\textrm{NC}}(2)\rightarrow\osr$ and $\iota_\ost:\mbox{\textrm{NC}}(2)\rightarrow\ost$ and ucp maps
$q_\osr:\osr\rightarrow\mbox{\textrm{NC}}(2)$ and $q_\ost:\ost\rightarrow\mbox{\textrm{NC}}(2)$ such that $q_\osr\circ\iota_\osr$
and $q_\ost\circ\iota_\ost$ are identity maps on $\osr$ and $\ost$. If there were to exist a unital complete order embedding
$\gamma:\osr\oc\ost \rightarrow\osr\omax\ost$, then $(q_\osr\otimes q_\ost)\circ\gamma\circ(\iota_\osr\otimes\iota_\ost)$
would be a unital complete order embedding 
$\mbox{\textrm{NC}}(2)\oc\mbox{\textrm{NC}}(2)\rightarrow\mbox{\textrm{NC}}(2)\omax\mbox{\textrm{NC}}(2)$, which is in contradiction to
\cite[Corollary 7.12]{farenick--kavruk--paulsen--todorov2014}. Hence, it must be that 
$\osr\oc\ost \not = \osr\omax\ost$.
\end{proof}

\section{Applications and Additional Properties}

\subsection{Scaling constants for minimal and maximal noncommutative cubes and prisms}
If, for a compact convex subset $K\subset\mathbb C^d$, there exists a
real number $C>0$ such that
\[
K^{\textrm{max}} \,\subseteq\,C\cdot K^{\textrm{min}},
\]
then $C$ is called a \emph{scaling constant} for $K$. 

While not all compact convex sets $K$ admit a scaling constant, the classical prisms and cubes do.
To see this, we draw
on the work of Passer, Shalit, and Solel in \cite{passer--shalit--solel2018}; in particular, they 
showed that any compact convex body $K\subseteq\mathbb R^d$ containing $0\in\mathbb R^d$ 
in the interior of $K$ admits a minimal scaling constant.

For every $d\geq 2$, the minimal scaling constant for $[-1,1]^d$ is exactly $\sqrt{d}$ \cite[p.~3233]{passer--shalit--solel2018}.
Let $\theta_k$ denote the minimal scaling constant for $\mbox{\textrm P}(k)$. We aim to find a lower bound for $\theta_k$.

By \cite[Proposition 3.4]{passer--shalit--solel2018}, if $K,L\subseteq\mathbb R^d$ are compact convex bodies, then
\[
\theta(K)\leq\rho(K,L)\theta(L),
\]
where $\theta(K)$ and $\theta(L)$ are the minimal scaling constants for $K$ and $L$, and 
\[
\rho(K,L)=\inf\{D>0\,|\, \exists\, R\in\mbox{\textrm GL}_d(\mathbb R)\mbox{ such that } K\subseteq R(L)\subseteq D\cdot K\}.
\]
We shall use these quantities above for the convex bodies 
$K=\mathbb B^{\mathbb R^3}$, the closed Euclidean unit ball in $\mathbb R^3$, and $L=\mbox{\textrm P}(k)$.
The minimal
scaling constant for $\mathbb B^{\mathbb R^3}$ is exactly $d$ \cite{passer--shalit--solel2018}, and so we seek an upper bound on
$\rho(\mathbb B^{\mathbb R^3},\mbox{\textrm P}(k))$ in the inequality
\[
d \leq \rho(\mathbb B^{\mathbb R^3},\mbox{\textrm P}(k))\,\theta_k.
\]

To this end, note that if $\xi\in\mbox{\textrm P}(k)$, then $\xi\in \overline{\mathbb D}\times[-1,1]$ and, therefore, 
$\|\xi\|\leq\sqrt{2}$. Thus,
\[
\mbox{\textrm P}(k) \,\subseteq\,\sqrt{2} \cdot \mathbb B_1^{\mathbb R^3}.
\]
Because the origin $0\in\mathbb R^3$ is interior to both $\mathbb B^{\mathbb R^3}$ and $\mbox{\textrm P}(k)$,
there is a maximal $r_k>0$ for which
\[
r_k\cdot \mathbb B^{\mathbb R^3} \subseteq \mbox{\textrm P}(k).
\]
By projecting onto the plane determined by the first two coordinates of $\mathbb R^3$, 
we see that the projection of $r_k\cdot \mathbb B^{\mathbb R^3}$
onto this plane must lie within the circle of radius $s_k$ inscribed by the convex hull 
$\mbox{\textrm Conv}\,C_k$ of the $k$-th roots of unity and centred at the origin. Thus, $r_k\leq s_k$. 
Since the incircle of $\mbox{\textrm Conv}\,C_k$ is tangent to the sides of the regular polygon $\mbox{\textrm P}(k)$
and touches the sides exactly at the midpoint of sides, 
the radius $s_k$ of the incircle equals the length of the midpoint of the line segment 
joining $1$ and $\omega$.
The midpoint of line segment 
joining $1$ and $\omega$ is $\frac{1}{2}(1+\cos(2\frac{\pi}{k}),\sin(2\frac{\pi}{k}))$, with length equal to $\frac{1}{2}\sqrt{2+2\cos(2\frac{\pi}{k})}
=\cos(\frac{\pi}{k})$.
Hence, 
\[
r_k\leq s_k=\cos\left(\frac{\pi}{k}\right).
\] 
Furthermore, if $\xi=(\xi_1,\xi_2,\xi_3)\in \mathbb R^3$ with $\|\xi\|\leq \cos\left(\frac{\pi}{k}\right)$, 
then $|\xi_3|\leq \cos\left(\frac{\pi}{k}\right)\leq 1$, and $\|(\xi_1,\xi_2)\|\leq \cos\left(\frac{\pi}{k}\right)$. 
Since the radius of the incircle in $\mbox{\textrm Conv}\,C_k$ exactly equals $\cos\left(\frac{\pi}{k}\right)$, 
it follows that $(\xi_1,\xi_2)\in \mbox{\textrm Conv}\,C_k$ and, hence, that $\xi\in \mbox{\textrm P}(k)$. Thus,   
\[
r_k=s_k=\cos\left(\frac{\pi}{k}\right).
\]
From
\[
\mathbb B^{\mathbb R^3} \subseteq \frac{1}{r_k}\cdot \mbox{\textrm P}(k) \subseteq \frac{\sqrt{2}}{r_k}\cdot \mathbb B^{\mathbb R^3},
\]
we deduce $\rho(\mathbb B^{\mathbb R^3},\mbox{\textrm P}(k))\leq \frac{\sqrt{2}}{r_k}$. Thus,
\[
\frac{3}{\sqrt{2}}\cos\left(\frac{\pi}{k}\right) \leq \frac{3}{\sqrt{2}}r_k 
\leq \frac{3}{\rho(\mathbb B^{\mathbb R^3},\mbox{\textrm P}(k))}\leq \theta_k.
\]
Thus, for the classical triangular prism $\mbox{\textrm P}(3)$, we have 
\[
\theta_3\geq \frac{3}{2\sqrt{2}},
\]
resulting in the following variant of the Halmos-Mirman dilation theorem (Theorem \ref{halmos-mirman}).

\begin{theorem}[Scaled Halmos-Mirman Theorem] 
If $x,y\in\B(\H)$ are such that $x$ has numerical range contained in the triangle $\mbox{\textrm Conv}\,C_3$
and $y$ is a selfadjoint contraction, then there are commuting unitaries $u$ and $v$ acting on a 
Hilbert space $\K$ containing $\H$ as a subspace, and a constant
$C\geq \frac{3}{2\sqrt{2}}$, such that $u^3=v^2=1$  and
$Cu$ and $Cv$ are joint dilations of $x$ and $y$ to $\K$. 
\end{theorem}
 
 \subsection{Exactness and the lifting property}
 
 A finite-dimensional operator system $\osr$ is \emph{exact} if, for every unital C$^*$-algebra $\A$ and ideal $\jay\subseteq\A$,
 \[
( \osr\omin\A/( \osr\omin\jay) \cong_{\textrm{OpSys}} \osr\omin (\A / \jay).
 \]
 A finite-dimensional operator system $\osr$ has the \emph{lifting property} if every ucp map $\osr\rightarrow\A/\jay$
 has a ucp lift to $\A$, for every unital C$^*$-algebra $\A$ and ideal $\jay\subseteq\A$. 
 
The
relation between exactness and the lifting property for finite-dimensional operator systems
is given by a theorem of Kavruk \cite[Theorem 6.6]{kavruk2014}: a finite-dimensional operator system
$\osr$ is exact if and only if its operator system dual $\osr^\delta$ has the lifting property.

\begin{theorem}\label{lifting exactness} If $\osr\in\{\mbox{\textrm{NC}}(d), \mbox{\textrm{NCP}}(k)\,|\,d,k\geq 3\}$, 
then $\osr$ has the lifting property, but is not an exact operator
system.
\end{theorem}

\begin{proof} The operator system $\osr$ has the lifting property because, by Theorem \ref{omax},
$\osr$ is an OMAX operator system \cite[Lemma 9.10]{kavruk2014}. To show $\osr$ is not exact, recall that 
$\cstar(\ftwo)$ admits an embedding into both $\cstar(\mathbb Z_k\ast\mathbb Z_2)$  and $\cstar(\ast_1^d\mathbb Z_2)$
because 
$\ftwo$ is a subgroup of both $*_1^d \mathbb Z_2$ and $\mathbb Z_k * \mathbb Z_2$. As $\cstar(\mathbb F_2)$ 
is not an exact C$^*$-algebra \cite{wassermann1990}, neither are 
$\cstar(\mathbb Z_k\ast\mathbb Z_2)$ or $\cstar(\ast_1^d\mathbb Z_2)$, since exactness passes to subalgebras.  
In other words, $\cstare(\osr)$ is not exact. Because $\cstare(\osr)$ is generated by the unitaries in $\osr$,
the operator system $\osr$ cannot be exact \cite[Corollary 9.6]{kavruk--paulsen--todorov--tomforde2013}.
\end{proof}

In relation to the problem of finding low-dimensional operator systems lacking the lifting property \cite{harris2025,scherer2026}, 
we note the following
result.

\begin{corollary}
	The $4$-dimensional operator system
\[
\mathcal S_{3,2}=\{(z_1,z_2,z_3,z_4,z_5)\in\mathbb C^5 \,|\, z_1+z_2+z_3=z_4+z_5\},
\]
with Archimedean order unit $e_{3,2}=(\frac{1}{3}, \frac{1}{3}, \frac{1}{3}, \frac{1}{2}, \frac{1}{2})$ and matrix ordering induced by the
inclusion of $\mathcal S_{3,2}$ into $\mathbb C^5$,
is exact, but does not have the lifting property.
\end{corollary}
\begin{proof}
Theorem \ref{dual ncp(k)} shows, with respect to the Archimedean order unit $e_{3,2}$ for $\mathcal S_{3,2}$, that
$\mathcal S_{3,2}\cong_{\textrm{OpSys}}\mbox{\textrm{NCP}}(3)^\delta$, for an 
appropriate choice of order unit for $\mbox{\textrm{NCP}}(3)^\delta$. As mentioned earlier, because
$\mbox{\textrm{NCP}}(3)=\mbox{\textrm{NCP}}(3)^{\textrm{max}}$, the operator system $\mbox{\textrm{NCP}}(3)$ has the lifting property; 
therefore, the operator system 
dual $\mbox{\textrm{NCP}}(3)^\delta$ is exact.
Hence, $\mathcal S_{3,2}$ is an
exact operator system.

Similarly, if $\mathcal S_{3,2}$ were to have the lifting property, then
its operator system dual would be exact; that is,
$\mbox{\textrm{NCP}}(3)$ would be exact. However, by Theorem \ref{lifting exactness}, 
$\mbox{\textrm{NCP}}(3)$ is not an exact operator system.
\end{proof}

Similar arguments, together with \cite[Theorem 5.9]{farenick--kavruk--paulsen--todorov2018}, yield:

\begin{corollary}
For every $k,d\geq 3$ and appropriate choices of Archimedean order units and matrix orderings induced by inclusions
into $\mathbb C^d$ and $\mathbb C^{k+2}$, the operator systems  
\[
\mathcal R_d=\{(\alpha_1+\beta_1, \alpha_1-\beta_1,\dots,\alpha_d+\beta_d,\alpha_d-\beta_d) \in\mathbb C^{2d}\,|\, \alpha_j,\beta_j\in\mathbb C, j=1,\dots,d\}
\]
and
\[
\mathcal S_{k,2}=\{(z_1,z_2,\dots,z_{k+2})\in\mathbb C^{k+2} \,|\, z_1+\cdots+z_{k}=z_{k+1}+z_{k+2}\},
\]
are exact, but
do not have the lifting property.
\end{corollary}

 \subsection{Controlling the codomain of ucp extensions}
 
 If one has an injective von Neumann algebra $\mathcal N$ and an inclusion $\osr\subseteq\ost$ of operator systems,
 then every ucp map $\phi:\osr\rightarrow\mathcal N$ extends to a ucp map $\Phi:\ost\rightarrow\mathcal N$. However, usually there
 is little one can say about the range of an extension $\Phi$ relative to the range of $\phi$. Interestingly, noncommutative cubes and prisms
 admit ucp extensions to the C$^*$-envelopes in which the codomain of the extension is the von Neumann algebra generated by the range of
 the original map.
 
 \begin{theorem}\label{control} If $\osr\in\{\mbox{\textrm{NC}}(d),\mbox{\textrm{NCP}}(k)\,|\,d,k\in\mathbb N\}$, and if $\phi:\osr\rightarrow\B(\H)$
 is a ucp map, then there exists a ucp map $\Phi:\cstare(\osr)\rightarrow\phi(\osr)^{''}$, the von Neumann algebra generated by
 $\phi(\osr)$, such that $\Phi_{\vert\osr}=\phi$.
 \end{theorem}
 
 \begin{proof} Under the stated hypothesis, 
 $\osr\oc\ost\subseteq_{\textrm{coi}}\cstare(\osr)\omax\ost$, for every operator system $\ost$, 
 by Lemma \ref{rdci_prism}. Therefore, the assertion about ucp extensions follows from 
 \cite[Theorem 4.1(3)]{bhattacharya2014}.
 \end{proof}
 
 A general property an operator system $\osr$ may have, related to the property exhibited in Theorem \ref{control}, is
 the \emph{double commutant expectation property} \cite{kavruk--paulsen--todorov--tomforde2013}, which is the 
 property that, for every unital complete order embedding 
 $\kappa:\osr\rightarrow\B(\H)$, there exists a ucp map $\phi:\B(\H)\rightarrow\kappa(\osr)^{''}$ such that $\phi\circ\kappa=\kappa$.
 
 \begin{theorem}\label{no dcep} If $\osr\in\{\mbox{\textrm{NC}}(d),\mbox{\textrm{NCP}}(k)\,|\,d,k\geq3\}$, then $\osr$ does not
 have the double commutant expectation property.
 \end{theorem}
 
 \begin{proof} By Theorem \ref{lifting exactness}, $\osr$ has the lifting property. Therefore, for such an operator system, the double
 commutant expectation property is equivalent, by \cite[Section 4.4]{kavruk2014}, to the property that 
\[
\mathcal R\otimes_{\textrm{el}}\mathcal T=\mathcal R\otimes_{\textrm c}\mathcal T,
\]
for every operator system $\ost$, where $\osr\otimes_{\textrm{el}}\ost$ denotes the left injective tensor product 
of $\osr$ and $\ost$ induced by the canonical inclusion of $\osr\otimes\ost$ into $\mathcal I(\osr)\omax\ost$.

With the right injective tensor product $\osr\otimes_{\textrm{er}}\ost$
of $\osr$ and $\ost$ induced by the canonical inclusion of $\osr\otimes\ost$ into $\osr\omax\mathcal I(\ost)$,
the fact that $\osr$ has the lifting property implies, by \cite[Theorem 4.3]{kavruk2014}, that
\[
\mathcal R\otimes_{\textrm{er}}\mathcal R=\mathcal R\otimes_{\textrm{min}}\mathcal R.
\]
Thus, by exchanging tensor factors, $\osr\otimes_{\textrm{el}}\osr=\osr\omin\osr$. 
Therefore, if $\osr$ had the double commutant expectation property,
then we would have both $\osr\otimes_{\textrm{el}}\osr=\osr\oc\osr$
and $\osr\otimes_{\textrm{el}}\osr=\osr\omin\osr$, in contradiction to $\osr\omin\osr\not=\osr\oc\osr$.
 \end{proof}

 \subsection{Minimal C$^*$-covers and automatic complete positivity with maximal tensor products}
  
 \begin{theorem} If $\osr_j\in\{\mbox{\textrm{NC}}(d),\mbox{\textrm{NCP}}(k)\,|\,d,k\geq 3\}$, for $j=1,2$, 
 then
 \[
 \cstare(\osr_1\omax\osr_2) \not= \cstare(\osr_1)\omax\cstare(\osr_2).
 \]
 \end{theorem}
 
 \begin{proof} As defined in \cite{farenick--kavruk--paulsen--todorov2014}, the operator system tensor product structure
 $\otimes_{\textrm ess}$ on the algebraic tensor product $\osr_1\otimes\osr_2$ is the one induced by the canonical inclusion of
 $\osr_1\otimes\osr_2$ in $\cstare(\osr_1)\omax\cstare(\osr_2)$. 
 By Lemma \ref{rdci_prism}, we have that 
 \[
 \osr_1\oc\osr_2\subseteq_{\textrm{coi}}\cstare(\osr_1)\oc\osr_2\subseteq_{\textrm{coi}}\cstare(\osr_1)\omax\cstare(\osr_2),
 \]
 and so $\osr_1\otimes_{\textrm ess}\osr_2=\osr_1\oc\osr_2$. By \cite[Theorem 3.5]{farenick--kavruk--paulsen--todorov2014}, for any operator system
 tensor product structure $\otimes_\sigma$ such that $\oss\otimes_\sigma\ost\subseteq\oss\otimes_{\textrm ess}\ost$, for all $\oss$ and $\ost$,
 $\osr_1\otimes_{\textrm ess}\osr_2=\osr_1\otimes_\sigma\osr_2$ 
  if and only if $\cstare(\osr_1\otimes_\sigma\osr_2) = \cstare(\osr_1)\omax\cstare(\osr_2)$. Because
   $\osr_1\oc\osr_2\not=\osr_1\omax\osr_2$, we deduce $\cstare(\osr_1\omax\osr_2) \not= \cstare(\osr_1)\omax\cstare(\osr_2)$, as asserted.
 \end{proof}
 
 It is worth noting another outcome of Theorem \ref{main result tp}, as it illustrates the loss of algebraic information when passing to tensor products.
 
 \begin{corollary} If
 $\osr_j\in\{\mbox{\textrm{NC}}(d),\mbox{\textrm{NCP}}(k)\,|\,d,k\geq 3\}$, for $j=1,2$, and if $u_j\in\osr_j$ is one of the 
 canonical unitary generators of $\cstare(\osr_j)$, then neither $u_1\otimes 1$ nor $1\otimes u_2$ is a unitary element of
 $\cstare(\osr_1\omax\osr_2)$.
 \end{corollary}
 
 \begin{proof} By \cite[Theorem 3.5]{farenick--kavruk--paulsen--todorov2014}, 
 $ \cstare(\osr_1\omax\osr_2) \not= \cstare(\osr_1)\omax\cstare(\osr_2)$ implies the asserted claims regarding 
 $u_1\otimes 1$ and $1\otimes u_2$.
 \end{proof}
  
 Lastly, we note:
 
 \begin{theorem}\label{acp tp}
 If $\osr_j\in\{\mbox{\textrm{NC}}(d),\mbox{\textrm{NCP}}(k)\,|\,d,k\geq 3\}$, then every positive linear map on $\osr_1\omax\osr_2$
 is completely positive. That is, 
 \[
 \osr_1\omax\osr_2 = (\osr_1\omax\osr_2)^{\textrm{max}}.
 \]
 \end{theorem}
 
 \begin{proof} 
 As shown in \cite{farenick--kavruk--paulsen--todorov2014,farenick--kavruk--paulsen--todorov2018}, 
 the operator systems $\mbox{\textrm{NCP}}(k)$ arise from complete quotient maps on $\mathbb C^{k+2}$,
 while the operator systems $\mbox{\textrm{NC}}(d)$ arise from complete quotient maps on $\mathbb C^{d}$. Thus, if 
 $\osr_j\in\{\mbox{\textrm{NC}}(d),\mbox{\textrm{NCP}}(k)\,|\,d,k\geq 3\}$ and $\phi_j:\mathbb C^{n_j}\rightarrow\osr_j$ are the complete quotient maps
 given in \cite{farenick--kavruk--paulsen--todorov2014,farenick--kavruk--paulsen--todorov2018}, for $j=1,2$,
 then $\phi_1\otimes\phi_2$ is a complete
 quotient map of $\mathbb C^{n_1}\omax\mathbb C^{n_2}$ onto $\osr_1\omax\osr_2$ \cite[Corollary 1.13]{farenick--paulsen2012}.
 Thus, positive linear maps $\psi:\osr_1\omax\osr_2\rightarrow\ost$ induce positive linear maps 
 $\psi\circ(\phi_1\otimes\phi_2)$ on $\mathbb C^{n_1}\omax\mathbb C^{n_2}$. Since
 $\mathbb C^{n_1}\omax\mathbb C^{n_2}=\mathbb C^{n_1}\omin\mathbb C^{n_2}$ is an abelian C$^*$-algebra, such positive
 linear maps on $\mathbb C^{n_1}\omax\mathbb C^{n_2}$ are completely positive. Because $\phi_1\otimes\phi_2$ is a complete quotient map,
 strictly positive matrices over $\osr_1\omax\osr_2$ lift to strictly positive matrices over $\mathbb C^{n_1}\omax\mathbb C^{n_2}$, implying
$\psi$ is completely positive.
 \end{proof}

\section*{Acknowledgements}
We thank Rajarama Bhat for providing us with the
linear-algebraic argument used to start the proof of Theorem \ref{irr repn}, 
complementing our group-theoretic arguments. Vern Paulsen pointed out to us
that coproducts could be used to give a proof of Theorem \ref{omax} that is cleaner than our original. 
We are also indebted to Remus Floricel and Allen Herman for their insights at an early stage of this work.
 
This work was partially supported by
the Discovery Grant program of
the Natural Sciences and Engineering Research Council of Canada and 
the Pacific Institute for the Mathematical Sciences Postdoctoral Fellowship program.



\begin{thebibliography}{99}


\bibitem{argerami--farenick2008}
Mart{\'{\i}}n Argerami and Douglas Farenick, Local multiplier algebras,
  injective envelopes, and type {I} {$W^*$}-algebras. \emph{J. Operator Theory}
  \textbf{59} (2008), no.~2, 237--245. 

\bibitem{argerami--farenick--massey2009}
Mart\'{i}n Argerami, Douglas Farenick, and Pedro Massey,  The gap between
  local multiplier algebras of {$C^*$}-algebras. \emph{Quarterly. J. Math. (Oxford)} \textbf{60}
  (2009), no.~3, 273--281. 

\bibitem{arveson1969}
William Arveson,  Subalgebras of {$C\sp{\ast} $}-algebras. \emph{ Acta Math.}
  \textbf{123} (1969), 141--224. 

\bibitem{bedos_omland2011}
Erik B\'edos and Tron~\AA. Omland, Primitivity of some full group
  {$C^\ast$}-algebras. \emph{Banach J. Math. Anal.} \textbf{5} (2011), no.~2, 44--58.

\bibitem{bhattacharya2014}
Angshuman Bhattacharya, Relative weak injectivity of operator system
  pairs. \emph{J. Math. Anal. Appl.} \textbf{420} (2014), no.~1, 257--267.

\bibitem{binding--farenick--li1995}
Paul Binding, Douglas Farenick, and Chi-Kwong Li,  A dilation and norm in
  several variable operator theory.  \emph{Canad. J. Math.} \textbf{47} (1995), no.~3,
  449--461. 

\bibitem{Blackadar-book}
B.~Blackadar, \emph{Operator algebras}, Encyclopaedia of Mathematical Sciences,
  vol. 122, Springer-Verlag, Berlin, 2006, Theory of $C\sp *$-algebras and von
  Neumann algebras, Operator Algebras and Non-commutative Geometry, III.
  


\bibitem{blecher--mcclure2026}
David~P. Blecher and Caleb~Becker McClure,  Real noncommutative convexity
  {I}. \emph{Studia Math.} \textbf{288} (2026), no.~2, 155--199. 

\bibitem{choi--effros1977}
Man~Duen Choi and Edward~G. Effros,  Injectivity and operator spaces. \emph{J.
  Functional Analysis} \textbf{24} (1977), no.~2, 156--209. 

\bibitem{davidson--kennedy2019}
Kenneth Davidson and Matthew Kennedy, Noncommutative {C}hoquet {T}heory.
\emph{Mem. Amer. Math. Soc.} \textbf{316} (2025), no.~1608, v+83. 

\bibitem{davisdon--dor-on--shalit--solel2017}
Kenneth~R. Davidson, Adam Dor-On, Orr~Moshe Shalit, and Baruch Solel,
Dilations, inclusions of matrix convex sets, and completely positive
  maps. \emph{Int. Math. Res. Not.} (2017), no.~13, 4069--4130. 

\bibitem{davis1955}
Chandler Davis,  Generators of the ring of bounded operators. \emph{Proc. Amer.
  Math. Soc.} \textbf{6} (1955), 907--972. 

\bibitem{evert--helton--klep--mccullough2018}
Eric Evert, J.~William Helton, Igor Klep, and Scott McCullough,  Extreme
  points of matrix convex sets, free spectrahedra, and dilation theory. \emph{J.
  Geom. Anal.} \textbf{28} (2018), no.~2, 1373--1408. 

\bibitem{farenick2000}
Douglas Farenick, Extremal matrix states on operator systems. \emph{J. London
  Math. Soc. (2)}  \textbf{61} (2000), no.~3, 885--892. 

\bibitem{farenick--kavruk--paulsen--todorov2014}
Douglas Farenick, Ali~S. Kavruk, Vern~I. Paulsen, and Ivan~G. Todorov,
Operator systems from discrete groups. \emph{Comm. Math. Phys.} \textbf{329}
  (2014), no.~1, 207--238. 

\bibitem{farenick--kavruk--paulsen--todorov2018}
Douglas Farenick, Ali~S. Kavruk, Vern~I. Paulsen, and Ivan~G. Todorov, 
Characterisations of the weak expectation property. \emph{New York J.
  Math.} \textbf{24A} (2018), 107--135. 

\bibitem{farenick--paulsen2012}
Douglas Farenick and Vern~I. Paulsen,  Operator system quotients of matrix
  algebras and their tensor products. \emph{Math. Scand.} \textbf{111} (2012), no.~2,
  210--243. 

\bibitem{farenick--tessier2022}
Douglas Farenick and Ryan Tessier, Purity of the embeddings of operator
  systems into their ${\textrm C}^*$- and injective envelopes. \emph{Pacific J. Math.}
  \textbf{317} (2022), no.~2, 317--338. 

\bibitem{fulton-harris-book}
William Fulton and Joe Harris, \emph{Representation theory}, Graduate Texts in
  Mathematics, vol. 129, Springer-Verlag, New York, 1991, A first course,
  Readings in Mathematics. 

\bibitem{goldbring2022}
Isaac Goldbring,  The {C}onnes embedding problem: a guided tour. \emph{Bull.
  Amer. Math. Soc. (N.S.)} \textbf{59} (2022), no.~4, 503--560. 

\bibitem{halmos1950}
Paul~R. Halmos,  Normal dilations and extensions of operators. \emph{Summa
  Brasil. Math.} \textbf{2} (1950), 125--134. 

\bibitem{hamana1979b}
Masamichi Hamana, Injective envelopes of operator systems. \emph{Publ. Res.
  Inst. Math. Sci.} \textbf{15} (1979), no.~3, 773--785. 

 \bibitem{harris2025}
Samuel~J. Harris,  Four-dimensional operator systems without the lifting
  property. 2025. \emph{arxiv 2508.00113} 
  
\bibitem{nonCEP}
Zhengfeng Ji, Anand Natarajan, Thomas Vidick, John Wright, and Henry Yuen,
{MIP}{${}\sp*$}={RE}. 2021. \emph{arxiv 2001.04383}

\bibitem{kavruk2012}
Ali~S. Kavruk, The weak expectation property and {R}iesz interpolation. 2012.
\emph{arxiv 1201.1514}

\bibitem{kavruk2014}
Ali~S. Kavruk, Nuclearity related properties in operator systems. \emph{J. Operator
  Theory} \textbf{71} (2014), no.~1, 95--156. 

\bibitem{kavruk--paulsen--todorov--tomforde2011}
Ali~S. Kavruk, Vern~I. Paulsen, Ivan~G. Todorov, and Mark Tomforde,
Tensor products of operator systems. \emph{J. Funct. Anal.} \textbf{261}
  (2011), no.~2, 267--299. 

\bibitem{kavruk--paulsen--todorov--tomforde2013}
Ali~S. Kavruk, Vern~I. Paulsen, Ivan~G. Todorov, and Mark Tomforde, 
Quotients, exactness, and nuclearity in the operator system
  category. \emph{ Adv. Math.} \textbf{235} (2013), 321--360. 

\bibitem{kirchberg1993}
Eberhard Kirchberg,  On nonsemisplit extensions, tensor products and
  exactness of group {$C^*$}-algebras. \emph{Invent. Math.} \textbf{112} (1993),
  no.~3, 449--489. 

\bibitem{kirchberg--wassermann1998}
Eberhard Kirchberg and Simon Wassermann, {$C^\ast$}-algebras generated by
  operator systems. \emph{J. Funct. Anal.} \textbf{155} (1998), no.~2, 324--351.

\bibitem{kriel2019}
Tom-Lukas Kriel, An introduction to matrix convex sets and free
  spectrahedra. \emph{Complex Anal. Oper. Theory} \textbf{13} (2019), no.~7,
  3251--3335. 

\bibitem{miller1928}
George~Abram Miller,  Possible orders of two generators of the alternating
  and of the symmetric group. \emph{Trans. Amer. Math. Soc.} \textbf{30} (1928),
  no.~1, 24--32. 

\bibitem{mirman1968}
B.~A. Mirman, The numerical range of a linear operator, and its norm.
\emph{Vorone\v z. Gos. Univ. Trudy Sem. Funkcional. Anal.} (1968), no.~10, 51--55.
  \MR{417814}

\bibitem{paschke--salinas1979}
William~L. Paschke and Norberto Salinas,  {$C\sp{\ast} $}-algebras
  associated with free products of groups. \emph{Pacific J. Math.} \textbf{82}
  (1979), no.~1, 211--221. 

\bibitem{passer--shalit--solel2018}
Benjamin Passer, Orr~Moshe Shalit, and Baruch Solel,  Minimal and maximal
  matrix convex sets. \emph{J. Funct. Anal.} \textbf{274} (2018), no.~11, 3197--3253.

\bibitem{paulsen--todorov--tomforde2010}
Vern~I. Paulsen, Ivan~G. Todorov, and Mark Tomforde,  Operator system
  structures on ordered spaces. \emph{Proc. Lond. Math. Soc. (3)} \textbf{102}
  (2011), no.~1, 25--49. 

\bibitem{pedersen1978}
Gert~K. Pedersen,  Approximating derivations on ideals of
  {$C\sp*$}-algebras. \emph{Invent. Math.} \textbf{45} (1978), no.~3, 299--305.

\bibitem{raeburn--sinclair1989}
Iain Raeburn and Allan~M. Sinclair, The {$C^*$}-algebra generated by two
  projections. \emph{Math. Scand.} \textbf{65} (1989), no.~2, 278--290. 

\bibitem{scherer2026}
Marcel Scherer,  A three-dimensional operator system without the
  {S}mith-{W}ard property. 2026 \emph{arxiv 2607.04274}

\bibitem{simon-book}
Barry Simon, \emph{Representations of finite and compact groups}, Graduate
  Studies in Mathematics, vol.~10, American Mathematical Society, Providence,
  RI, 1996. 

\bibitem{sunder1988}
V.~S. Sunder,  {$N$} subspaces. \emph{Canad. J. Math.} \textbf{40} (1988),
  no.~1, 38--54. 

\bibitem{wassermann1990}
Simon Wassermann, Tensor products of free-group {$C^*$}-algebras. \emph{Bull.
  London Math. Soc.} \textbf{22} (1990), no.~4, 375--380. 

\bibitem{webster--winkler1999}
Corran Webster and Soren Winkler, The {K}rein-{M}ilman theorem in
  operator convexity. \emph{Trans. Amer. Math. Soc.} \textbf{351} (1999), no.~1,
  307--322. 



\end{thebibliography}
\end{document}